\documentclass[12pt]{article}

\usepackage{layout, graphicx, amsmath, amsthm, amssymb, euscript, enumerate, bbold, bbm, graphicx, fancyhdr,xcolor,mathtools, tikz,url,mathrsfs}
\usepackage{float}
\usepackage{subcaption}
\usetikzlibrary{arrows,shapes}
\usepackage{verbatim}
\definecolor{lila}{RGB}{150,51,255}
\definecolor{verdesi}{RGB}{0,153,0}
\usepackage{yfonts}
\usetikzlibrary{calc,trees,positioning,arrows,chains,shapes.geometric,%
	decorations.pathreplacing,decorations.pathmorphing,shapes,%
	matrix,shapes.symbols,plotmarks,decorations.markings,shadows}
\usepackage{caption}
\usepackage{subcaption}
\usepackage{enumitem, enumerate}
\usepackage{ellipsis} 
\usepackage[pdftex,bookmarks,colorlinks]{hyperref}
\hypersetup{colorlinks=false}
\makeatletter                                         
\def\th@newremark{\th@remark\thm@headfont{\bfseries}}  
\makeatletter                                          

\theoremstyle{newremark}                               

\newcommand{\bare}[1]{\overline{\VEC{e}}^{\,#1}}  
\DeclareMathOperator{\inverseGauss}{IG}
\DeclareMathOperator{\Var}{Var}

\def\N{\mathbb{N}}
\def\P{\mathbb{P}}
\def\R{\mathbb{R}}
\def\E{\mathbb{E}}

\def\U{\mathcal{U}}

\def\ind{ \mathbbm{1} }

\newtheorem{theorem}{Theorem}[section]
\newtheorem{corollary}[theorem]{Corollary}
\newtheorem{definition}[theorem]{Definition}
\newtheorem{lemma}[theorem]{Lemma}
\newtheorem{proposition}[theorem]{Proposition}
\newtheorem{remark}[theorem]{Remark}

\numberwithin{equation}{section}

\newcommand{\VEC}[1]{\mathbf{#1}}

\definecolor{burgundy}{rgb}{0.5, 0.0, 0.13}
\definecolor{linkblue}{RGB}{0,20,128}
\definecolor{linkred}{RGB}{128, 0, 6}
\definecolor{citegreen}{RGB}{46, 126, 42}
\definecolor{urlmagenta}{RGB}{138, 0, 135}
\hypersetup{
	linkcolor  = linkblue,
	citecolor  = citegreen,
	urlcolor   = urlmagenta,
	colorlinks = true,
}

\makeindex
\begin{document}

\title{Scaling limits for allelic subfamilies in a mother-dependent multitype branching process}
	\author{  Airam Blancas\thanks{Department of Statistics, ITAM, Mexico.
		Email: {airam.blancas@itam.mx}},  Mar\'ia Clara Fittipaldi\thanks{Facultad de Ciencias, UNAM, Mexico. Email: {mcfittipaldi@ciencias.unam.mx} }, 
 Sara\'i Hern\'andez-Torres\thanks{Instituto de Matem\'aticas, UNAM, Mexico. Email: {saraiht@im.unam.mx} }}
	\date{}
	\maketitle
	\vspace{-.2in}

\begin{abstract}
The mother-dependent neutral mutations model describes the evolution of a population across discrete generations in which every individual carries one of $d$ alleles and each mutant child acquires a type different from that of its mother, chosen uniformly at random. We consider a critical population with finite offspring variance and analyze the genealogy of the allelic subfamilies, extending the framework of Bertoin~\cite{Bertoin10} from the infinite-allele setting to finitely many alleles.

We study the Markov chain that records the sizes of the allelic subfamilies of each type together with the numbers of mutants that they beget, and we give a $d$-dimensional random walk representation of it. In a regime of large initial population and rare mutations, we prove that the rescaled chain converges to a discrete-time continuous-state Markov process driven by the masses of the mutant populations; from the first allelic generation onwards its clone component is a multitype Ji\v{r}ina process whose branching mechanism is the Laplace exponent of a $d$-dimensional inverse Gaussian subordinator carried by the coordinate axes. Lifting this convergence to the genealogical level, we introduce a $d$-type tree of alleles and show that, in the scaling limit, it converges to a tree-indexed multitype Ji\v{r}ina process, described by the ranked atoms of Poisson random measures with inverse Gaussian intensity.

\medskip



\end{abstract}

\section{Introduction}

Branching processes with mutations provide a probabilistic framework for
studying the genealogy of genetic diversity in evolving populations. This
diversity arises from mutations, which produce distinct versions of a gene,
called \textit{alleles}. 
Reproduction governs the growth of the population, while
mutations partition it into \textit{allelic subfamilies}. 
Our aim is to describe the sizes of these subfamilies and the genealogical relations among them in
populations carrying finitely many alleles. This is a classical problem in
population genetics, lying at the interface of probability and mathematical
biology. 

A natural way to build such models is to start from Bienaym\'e-Galton-Watson
(BGW) processes. The extensions of BGW processes that are relevant here differ in the assumptions placed on the mutation mechanism.
A classical framework is Kimura's infinite-alleles model, in which every
mutation gives rise to a novel allele~\cite{Kimura1971}. 
Under this assumption, the genealogy of allelic subfamilies admits a description, developed by Bertoin in~\cite{Bertoin09, Bertoin10}. He showed that the pair recording the
size of a clone subfamily together with the number of mutant children it begets,
$((T_k, M_{k+1}) : k \in \N_0)$, is a Markov chain whose transition
probabilities are governed by its second coordinate, itself a one-type BGW
process. The genealogical relations between subfamilies are encoded by an
associated \textit{tree of alleles}. 
For a large initial population and rare
mutations, he established that the rescaled chain converges to a Ji\v{r}ina
process~\cite{jivrina1958}, that is, a discrete-time continuous-state branching
process, and that the rescaled tree of alleles converges to its tree-indexed
counterpart. In later work, Blancas and Rivero~\cite{BlancasRivero} extended
these results to offspring distributions with regularly varying tails.

There are two reasons to move beyond the infinite-alleles hypothesis. 
First, research in cancer genomics has cast doubt on its general validity~\cite{11GRAND, PMID:29030470}. 
Similar conclusions arise in population genetics, where evolutionary mechanisms have been observed in which the same mutation is acquired independently along distinct lineages~\cite{HU2017109}. 
These findings highlight the need for models that account for a finite number of alleles. 
Second, even when working under the infinite-sites assumption, computational treatments may operate with finitely many alleles. This is particularly relevant for current methods for simulation~\cite{SiFit17}, parameter inference~\cite{MSRM13, CellPhy}, and Bayesian phylogenetic reconstruction~\cite{SiCloneFit}. These considerations motivate the mother-dependent neutral mutations model studied in this work.

The \textit{mother-dependent neutral mutations model}, introduced in~\cite{BFH}, describes a population whose individuals
carry one of $d$ alleles, where $d \geq 2$ is fixed throughout.
It is governed by a
mutation probability $r$: independently of everything else, each child is a
\textit{clone} of its mother with probability $1-r$; otherwise it is a
\textit{mutant} and adopts one of the remaining $d-1$ alleles, uniformly at
random. 
The mutations are \textit{neutral}, in that they do not affect
reproduction: the total population size evolves as a one-type BGW process, and
the allele types are read off as a colouring of its genealogy, with the edges
marked at the mutation events (Figure~\ref{fig:stoppingline}). The formal definition of the model is given in Section~\ref{sec: model and results}.
Given that a mutant chooses a type different from its mother, types reappear in the lineage across time. This feature is absent in the infinite-alleles setting, and it is one of the main features that we aim to capture in the scaling limit studied in this work.

The population is naturally structured by \textit{allelic generations}: an
individual belongs to the $k$-th allelic generation when exactly $k$ mutations
have occurred along the lineage joining it to the original ancestor. Following
Bertoin's analysis of the infinite-alleles model~\cite{Bertoin10}, we study the
allelic subfamilies through two objects: a multidimensional Markov chain and a random tree. The Markov chain records, for each type $i \in [d]$ and each $k \in \N_0$, the number
$\VEC{T}_k(i)$ of type-$i$ individuals in the $k$-th allelic generation,
together with the number $\VEC{M}_{k+1}(i)$ of type-$i$ mutants founding
subfamilies in the next one. Collecting these into
$\VEC{T}_k, \VEC{M}_{k+1} \in \N_0^d$, the sequence of pairs
\[
    \bigl( (\VEC{T}_k, \VEC{M}_{k+1}) : k \in \N_0 \bigr)
\]
is the \textit{clone-mutant Markov chain}. The tree-based representation,
called the \textit{$d$-type tree of alleles} and defined in
Subsection~\ref{sec:alleleTree}, arranges the subfamilies in a rooted tree
whose vertices carry both a size and a type. Our main results describe the
scaling limits of both.

Our two main theorems concern a sequence of populations modeled by the mother-dependent neutral mutations, indexed by
$n \in \N$, in a critical and rare-mutation regime. We assume throughout that
the offspring law is critical with finite variance $\sigma^2$ (see Hypothesis~\ref{hyp:size}); that the initial
population is of order $n$, with $\VEC{a}^{(n)}(i) \sim n \VEC{y}(i)$ for a
fixed vector $\VEC{y}$; and that mutations are rare, with
$ r(n) \sim c\, n^{-1}$ for a constant $c > 0$ (assumptions~\ref{hyp:start} and~\ref{hyp:mutation}).  
Heuristically, the finite variance produces a diffusive scaling of the subfamily sizes, of
order $n^2$, while the rarity of mutations keeps the number of mutants per
allelic generation of order $n$.

Our first result, Theorem~\ref{thm:limclonemutchain}, identifies the limit of
the clone-mutant chain: as $n \to \infty$,
\begin{equation}
    \left( \bigl( n^{-2}\VEC{T}^{(n)}_k,\ n^{-1}\VEC{M}^{(n)}_{k+1} \bigr) : k \in \N_0 \right)
\overset{\mathcal{D}}{\Longrightarrow}
\bigl( (\VEC{X}_{k+1}, \VEC{Z}_{k+1}) : k \in \N_0 \bigr),
\end{equation}
where $\VEC{X}$ is a multitype Ji\v{r}ina process whose branching mechanism is governed by an
inverse Gaussian subordinator, and where $\VEC{Z}$ is the vector of mutant
masses.

Our second result, Theorem~\ref{thm:mainalles}, lifts the convergence to the
genealogical level: the rescaled $d$-type tree of alleles converges, in the
sense of finite-dimensional distributions, to a tree-indexed multitype
Ji\v{r}ina process. In the limit, a subfamily of type $i$ and mass $m$ begets,
for each type $j \neq i$, an infinite ranked sequence of subfamilies whose
masses are the atoms of a Poisson random measure with intensity
$(d-1)^{-1}m\,\nu$, where $\nu$ is the inverse Gaussian L\'evy measure
defined in~\eqref{eq: fdpinvgauss}. The total mass so created in type $j$ is inverse
Gaussian, and the number of mutants founding these subfamilies is, in the limit,
the deterministic quantity $\tfrac{c}{d-1}m$. Together, the two theorems give a
complete asymptotic description of the allelic subfamilies of the model, both
for their sizes and for the tree structure relating them.

The proofs combine three ingredients. First, we exploit the neutrality of the
mutations to view the mutation mechanism as a colouring of an underlying
one-type BGW genealogy (as explained in~\cite{BFH}). Second, we derive a random walk representation of the clone-mutant
chain, which extends the one-dimensional representation of~\cite{Bertoin10}.
Third, this representation lets us identify the scaling limits through
functional limit theorems.

The passage from infinite types to $d$ types requires a careful management of
the multitype setting. For instance, the mother-dependence constraint forces
every mutant to differ in type from its mother, so that under
$\P^{r}_{\VEC{e}_i}$ the $i$-th coordinate of $\VEC{M}_1$ vanishes.
In particular, the coordinates of $\VEC{M}_1$ are not exchangeable, and the
random walk representation has to be set up so as to respect this constraint. Additionally, the sizes of the subfamilies of a given allelic generation carry
types, and there is no canonical total order on the resulting labeled family;
we therefore index the tree of alleles by the Ulam--Harris tree over the
alphabet $[d]\times\N$, ranking subfamilies by size \textit{within each type}. This
bookkeeping, described in Subsection~\ref{sec:alleleTree}, is what makes the limiting object of Theorem~\ref{thm:mainalles} well defined.

\subsection*{Organization of the paper}

Section~\ref{sec: model and results} introduces the model and formally states the main results. Section~\ref{sec:bases} derives the distribution of $(\VEC{T}_0, \VEC{M}_1)$, characterizing the clone-mutant Markov chain. 
Section~\ref{sec: representations} presents two representations of allelic subfamilies. 
The first, in Subsection~\ref{subsec: rw}, is a random walk representation of the clone-mutant Markov chain, which is key to establishing Theorem~\ref{thm:limclonemutchain}. The second, in Subsection~\ref{sec:alleleTree}, is the $d$-type tree of alleles. 
Section~\ref{sec:proofmain1} establishes the scaling limits, proving Theorems~\ref{thm:limclonemutchain} and~\ref{thm:mainalles}. 
Finally, Appendix~\ref{sec:gralbrachprop} collects background material on the general branching property, and Appendix~\ref{app:proofsim} gathers the remaining technical proofs.

\section{Model and main results} \label{sec: model and results}

The objective of this section is to introduce the mother-dependent neutral
mutations model and the probabilistic objects that describe the genealogy of
its allelic subfamilies. We first define the underlying multitype branching
process and then introduce the clone-mutant Markov chain, which is the main
object of our analysis. Finally, we state the two principal scaling-limit
results concerning the clone-mutant chain and the associated multitype tree of
alleles.

\subsection{Notation}
\label{sec:notation}
 
We denote vectors in boldface, so $\VEC{v} \in \R^d$, and its $k$-th entry is  $\VEC{v}(k)$. Let $\VEC{e}_m  \in \R^d$ be the $m$-th unit vector of $\R^d$, meaning that $\VEC{e}_m (k) = \ind_m (k)$ for each $ k \in [d]$. We set $\N_0 \coloneqq \N \cup \{0\}$
and $\R_{+} \coloneqq [0,\infty)$, 
and for $\VEC{v} \in \R_{+}^d$ we write
$\vert\VEC{v}\vert \coloneqq \sum_{\ell=1}^d \VEC{v}(\ell)$.
Given $\VEC{v} \in \R^d$ and $i \in [d]$, let $\VEC{\bar{v}}^i$ denote the vector
obtained from $\VEC{v}$ by setting its $i$-th entry to zero,
\begin{equation} \label{eq:notationErase}
	\VEC{\bar{v}}^i \coloneqq (\VEC{v}(1),\dots,\VEC{v}(i-1),0,\VEC{v}(i+1),\dots,\VEC{v}(d)),
\end{equation}
and set $\bar{\N}_0^i \coloneqq \{\VEC{v} \in \N_0^d : \VEC{v}(i)=0\}$. We use
the product notation
\[
	\VEC{u}^{\VEC{v}} \coloneqq \prod_{\ell=1}^{d} \VEC{u}(\ell)^{\VEC{v}(\ell)},
	\qquad \VEC{u} \in [0,1]^d,\ \VEC{v} \in \N_0^d,
\]
with the convention $0^0=1$. We write $\VEC{v} \boldsymbol{\leq} \VEC{x}$ when
$\VEC{v}(\ell) \leq \VEC{x}(\ell)$ for every $\ell \in [d]$; this \textit{product
order} is a partial order on $\R^d$.

Finally, $f(n) \sim g(n)$ as $n \to \infty$ means that
$\lim_{n\to\infty} f(n)/g(n) = 1$; we write $\overset{\mathcal{D}}{=}$ for
equality in distribution and $\overset{\mathcal{D}}{\Longrightarrow}$ for
convergence in distribution; and $\mathcal{L}(X;\P)$ denotes the law of $X$
under $\P$.

\subsection{The mother-dependent neutral mutations model}

We consider a critical BGW population whose
individuals carry one of $d$ possible alleles.
Reproduction is independent of the allele type, while mutations occur
independently along genealogical lineages.
Whenever a mutation occurs, the offspring does not inherit the mother's allele
but instead adopts uniformly one of the remaining $d-1$ types.
Consequently, the population size evolves exactly as in the underlying BGW
process, whereas mutations determine the allelic composition of the population.

We now describe the model formally. Let  $\VEC{Y} \coloneqq (\VEC{Y}_k, k \in \mathbb{N}_0)$ be a $d$-type BGW process. For each $k\in\mathbb{N}_0$,  $\VEC{Y}_k$ describes the allelic configuration of the population in the $k$-th generation (also referred to as the $k$-th level), so that for each $i \in [d]$,  $\VEC{Y}_k(i)$ denotes the number of individuals of type $i$ in the $k$-th level. The process $\VEC{Y}$ is a discrete-time Markov chain, recursively defined  by
\begin{equation}\label{eq:multipopsize}
	\VEC{Y}_{0} =  \VEC{e}_i \qquad \text{and} \qquad  \VEC{Y}_{k+1}(j)  =  \sum_{i=1}^{d} \sum_{\ell=1}^{ \VEC{Y}_k (i)} \boldsymbol{\xi}_{k+1,\ell}^i{(j)} , \qquad  \text{ for all }   k \in \N_0,\, j\in [d],
\end{equation}
where $ \boldsymbol{\xi}^i_{k+1,\ell}(j)$ denotes the number of children of type $j$ in the $(k+1)$-th level begotten by the $\ell$-th individual of type $i$ in level $k$. 
The random vectors $(\boldsymbol{\xi}^i_{k,\ell})_{k \in \mathbb{N}_0 , \ell  \in \mathbb{N} }$ are independent copies of the random vector $ \boldsymbol{\xi}^i$. 
We denote by $\xi^{(+)}$ the \textit{total number of children}:
\begin{equation} \label{eq:offspringdist1}
	\xi^{(+)}\overset{\mathcal{D}}{=}\xi^{(+,i)}\coloneqq \sum_{j=1}^{d}\boldsymbol{ \xi}^{i}(j), \qquad i\in[d].
\end{equation} 
We assume that $\xi^{(+)}$ has distribution $\mu^{+}$, a probability measure supported  on $\mathbb{N}_0$ satisfying $ \mu^{+} (0) + \mu^{+} (1) < 1 $ to avoid trivial cases. 
We focus throughout on a critical population and impose the following assumption.

\begin{enumerate}[label=(\Alph*)]  
   \item \label{hyp:size}   
   The distribution $\mu^+$ of the total number of offspring of an  individual in the population is critical with finite variance $\sigma^2$, i.e.
    \[
 \E\left[\xi^{(+)} \right]=1 \qquad \text{ and } \qquad  
 \Var \left[\xi^{(+)} \right]\eqqcolon\sigma^2  \in (0, \infty).
	\]   
\end{enumerate}

Under the mother-dependence assumption, conditionally on $\{ \xi^{(+)}=\ell\}$, the random variable $\boldsymbol{\xi}^{i}(i)$ is $\mathrm{Binomial}(\ell, 1-r)$ and, for each $j \neq i$, $\boldsymbol{\xi}^{i}(j)$ is $\mathrm{Binomial}\left(\ell, \frac{r}{d-1}\right)$; jointly, $\boldsymbol{\xi}^i$ is multinomial with $\ell$ trials and probabilities $(1-r)$ for coordinate $i$ and $r/(d-1)$ for the others.

Formally, the mother-dependent neutral mutations model is the $d$-type BGW process $\VEC{Y}$ with offspring distribution $ \boldsymbol{\mu} = (\boldsymbol{\mu}_i, i\in [d])$, where $\boldsymbol{\mu}_i$ is given by
\begin{equation}\label{eq: mothdepproba}
	\begin{split}
		 \boldsymbol{\mu}_i(\VEC{v})  
		 &= \mathbb{P}_{\VEC{e}_i}( \boldsymbol{\xi}^{i}=\VEC{v}) = \mathbb{P}_{\VEC{e}_i}\left[ (\boldsymbol{\xi}^{i}{(1)},\dots,\boldsymbol{\xi}^{i}{(d)})=(\VEC{v}(1),\dots,\VEC{v}(d))\right]\\
		 &= \mu^{+}(|\VEC{v}|) \binom{|\VEC{v}|}{\VEC{v}(1),\dots,\VEC{v}(d)} (1-r)^{\VEC{v}(i)} \prod\limits_{j \neq i}\left(\dfrac{r}{d-1}\right)^{\VEC{v}(j)} . 
	\end{split} 
\end{equation}

We write $\{ (\VEC{Y}_k, k \in \mathbb{N}_0); \mathbb{P}_{\VEC{e}_i}^{r}\}$ to indicate that $\mathbf{Y}$ starts from $\mathbf{e}_i$ and that the model parameter is the mutation probability $r \in (0,1)$. 
In the situation where the starting population is arbitrary, we use the notation $\{ \VEC{Y} \coloneqq (\VEC{Y}_k, k \in \mathbb{N}_0); \mathbb{P}_{\VEC{a}}^{r}\}$, with $\VEC{a}=(\VEC{a}(1),\dots,\VEC{a}(d)) \in \mathbb{N}_0^d\setminus \{ \VEC{0} \}$. 

By the branching property, conditionally on $\VEC{Y}_k$, the subpopulations descending
from the individuals of the $k$-th generation are independent copies of the process
started from a single individual of the corresponding type, independent of
$(\VEC{Y}_0,\dots,\VEC{Y}_{k-1})$.
Therefore, $\mathbb{P}_{\VEC{a}}^{r}$ is the product measure
\begin{equation}\label{eq: product measure}
\P^r_{\VEC{a}} 
= \left(\P^r_{\VEC{e}_1}\right)^{*\VEC{a}(1)} \cdots \left(\P^r_{\VEC{e}_d}\right)^{*\VEC{a}(d)},
\end{equation}
where $*$ denotes the convolution product.

We shall be interested not only in the evolution of the population but also in
the genealogy induced by mutations. This leads naturally to the clone-mutant
Markov chain introduced below.

\subsection{The clone-mutant Markov chain} \label{sec:Markovchain}

The genealogy of a $d$-type BGW process $\{ (\VEC{Y}_k, k \in \mathbb{N}_0);\mathbb{P}_{\VEC{e}_i}^{r}\}$ can be represented by $\mathcal{T}_d$, a marked planar tree with colored vertices and marked edges, where the colors indicate the types of the individuals and the marks indicate a mutation, as in Figure \ref{fig:stoppingline}.

\tikzset{
		fannode/.style={circle, draw=white, fill=white, line width=3pt, minimum size=1cm},
		rorangenode/.style={circle, draw=black, fill=orange, line width=2pt, minimum size=1cm},
		orangenode/.style={circle, draw=black, fill=orange!25, line width=2pt, minimum size=1cm},
		rosanode/.style={circle, draw=black, fill=magenta!40, line width=2pt, minimum size=1cm},
		rrosanode/.style={circle, draw=black, fill=magenta, line width=2pt, minimum size=1cm},
		verdenode/.style={circle, draw=black, fill=verdesi!20, line width=2pt, minimum size=1cm},
		rverdenode/.style={circle, draw=black, fill=verdesi, line width=2pt, minimum size=1cm},
		lilanode/.style={circle, draw=black, fill=lila!25, line width=2pt, minimum size=1cm},
		rlilanode/.style={circle, draw=black, fill=lila, line width=2pt, minimum size=1cm}
	}
	\begin{figure}[ht]
    \centering
    \resizebox{.5\textwidth}{!}{%
	\begin{tikzpicture}[
		>=latex, glow/.style={%
			preaction={draw,line cap=round,line join=round,
				opacity=0.5,line width=15pt,#1}},glow/.default=green,
		transparency group,
		level 1/.style={sibling distance=69mm, level distance=20mm},
		level 2/.style={sibling distance=20mm},
		level 3/.style={sibling distance=12mm},
		level 4/.style={sibling distance=11mm},
		line/.style={edge from parent/.style={solid, magenta, ultra thick, draw}},
	        line2/.style={edge from parent/.style={solid, cyan, ultra thick, draw}}
		]
\node[rverdenode] { } 
	child{node[verdenode] {}
			child{node[verdenode] {} 
				child{node[rlilanode, glow] {} 
				edge from parent node {\textcolor{red}{\textbf{\large X}} } 
				}        
			}
			child{node[rorangenode, glow] {} 
				child{node[orangenode] {} }
			        child{node[orangenode] {} }
			edge from parent node {\textcolor{red}{\textbf{\large X}} }
			}
			child{node[rorangenode, glow] {}
				child{node[rverdenode, glow=pink] {} 
				edge from parent node {\textcolor{red}{\textbf{\large X}} } 
				}
			edge from parent node {\textcolor{red}{\textbf{\large X}} }
			}         
		}
	child{node[rorangenode, glow] {}
		child{node[rlilanode,glow=pink] {} 
		        child {node[rorangenode] {} 
		        edge from parent node {\textcolor{red}{\textbf{\large X}} }
               } 
			child {node[lilanode] {} }
		edge from parent node {\textcolor{red}{\textbf{\large X}} }
		} 
		child{node[orangenode] {}}
		child{node[rverdenode, glow=pink] {}
			child {node[verdenode] {} }
			edge from parent node {\textcolor{red}{\textbf{\large X}} }
		}
		child{node[rlilanode,glow=pink] {}  
		         child{node[lilanode] {} }
			 child{node[lilanode] {} }
		edge from parent node {\textcolor{red}{\textbf{\large X}} }
		}
	edge from parent node {\textcolor{red}{\textbf{\large X}} }
	};     
	\end{tikzpicture}}
    	 \caption{A representation of the mother-dependent neutral mutations model with $d = 3$ types. The types $1$, $2$ and $3$ are represented in green, orange, and purple, respectively. Mutations are indicated by marks on the edges, and mutant individuals are shown in a bolder shade. Mutants in the first allelic generation are highlighted in green, whereas those in the second one are in pink.}
	\label{fig:stoppingline}
\end{figure}

We say that an individual is in the \textbf{{$k$-th allelic}} \textit{generation} if it has $k$ marks in its ancestral line. We define $\VEC{T}_k(i)$ as the number of individuals of type $i$ in the $k$-th allelic generation. Those individuals with the $k$-th mark on the edge between them and their mother will be called {\textit{mutants of the $k$-allelic generation}}. We denote by $\VEC{M}_k{(i)}$ the number of 
mutants of type $i$ in the $k$-th allelic generation. 
By convention, individuals in the $0$-th level are
considered mutants of the $0$-th allelic generation. Hence, we have the $d$-dimensional random vectors:
\[
\mathbf{M}_k=( \VEC{M}_k{(1)}, \dots , \VEC{M}_k{(d)}) \qquad
\text{and} 
\qquad
\mathbf{T}_k=(\VEC{T}_k(1), \dots ,\VEC{T}_k(d)), \qquad k\in \mathbb{N}_0.
\]
By definition $\VEC{M}_k(i)\leq \VEC{T}_k(i)$, for all $i\in [d]$ and $k\in \N_0$.
If the process $\VEC{Y}$ starts with $\VEC{Y}_0=\VEC{e}_i$, we have
$\VEC{M}_0=\VEC{e}_i$, $\mathbb{P}_{\VEC{e}_i}^{r} \text{-a.s.}$

\begin{remark}\label{rem:mutdiftypmoth}
In the mother-dependent neutral mutations model, once a mutation occurs, the mutant offspring must differ in type from their mother. Consequently, mutants in the $(k+1)$-th allelic generation necessarily have a different type from their mother in the $k$-th allelic generation. In particular, if the starting population consists of a single individual of type $i$, then $\VEC{M}_1$ is necessarily zero in the $i$-th coordinate.
\end{remark}

We now observe that each mutant in the $k$-th allelic generation is the founder of its respective clone subfamily. Consequently, these individuals define a cut across the genealogical lines of the tree; in fact, they form a \textit{stopping line}. This concept was first introduced by Chauvin \cite{chauvin1986propriete}. Roughly speaking, a stopping line is a random family of vertices in the tree containing at most one vertex along each ancestral line, such that, for any given vertex, the event that it belongs to the line depends only on the marks observed along the path from the root to that vertex. A precise formulation can be found in Appendix~\ref{sec:gralbrachprop}, based on \cite{jagers2008}. According to Proposition \ref{prop: gralbranch}, the branching property extends to stopping lines; since this constitutes a stronger version of the classical branching property, it is referred to as the general branching property. As a straightforward consequence of the general branching property for the $d$-type BGW process $\VEC{Y}$, we obtain the following lemma.

\begin{lemma}\label{le:TyMMarkov}
For a mother-dependent neutral mutations model under $\mathbb{P}_{\VEC{e}_i}^{r}$,
$( \mathbf{M}_k: k \in\mathbb{N}_0 )$ is a $d$-type BGW process with reproduction law 
$\mathbb{P}^r_{\mathbf{e}_i}( \mathbf{M}_1\in \cdot)$. 

Moreover, $\left( (\mathbf{T}_k,\mathbf{M}_{k+1}): k \in \mathbb{N}_0   \right)$ is a  Markov chain satisfying
\begin{equation}\label{probatranTM} 
	\mathbb{P}_{\mathbf{e}_i}^r(\mathbf{T}_k =\mathbf{x},\mathbf{M}_{k+1}=\mathbf{y} \left| \right. \mathbf{T}_{k-1}=\mathbf{u},\mathbf{M}_{k}=\mathbf{v} ) 
	= 
	\mathbb{P}_{\mathbf{v}}^r (\mathbf{T}_0=\mathbf{x},\mathbf{M}_1=\mathbf{y}),
\end{equation}
for all $ \mathbf{u},\mathbf{v}, \mathbf{x}, \mathbf{y}\in\mathbb{N}_0^d$ with $\mathbf{v} \boldsymbol{\leq} \mathbf{x}$ and $\mathbb{P}_{\mathbf{e}_i}^r( \mathbf{T}_{k-1}=\mathbf{u},\mathbf{M}_{k}=\mathbf{v} ) >0$.
\end{lemma}

\subsection{The infinite-alleles case}

The mother-dependent neutral mutations model is a finite-allele variant of the
\textit{BGW process with neutral mutations} introduced by Bertoin
in~\cite{Bertoin09, Bertoin10}, and each of our results has a counterpart there.
We recall here only what is needed for the comparison, and refer
to~\cite{Bertoin10} for a complete account.

In the infinite-alleles model the population evolves as a one-type BGW process
in which, independently, each child is a clone of its mother with probability
$1-r$ and a mutant with probability $r$, and every mutation produces an allele that
has never occurred before. Writing $\xi^{(c)}$ and $\xi^{(m)}$ for the numbers
of clone and of mutant children of a typical individual, an individual begets
$j$ clones and $\ell$ mutants with probability
\begin{equation}\label{eq:offGWBer}
\P(\xi^{(c)}=j,\xi^{(m)}=\ell)=\mu^{+}(j+\ell)\binom{j+\ell}{j}(1-r)^j r^\ell,
\qquad j,\ell\in\N_0,
\end{equation}
where $\mu^{+}$ is again the law of the total number of children.
Note that $\xi^{(c)}$ is Binomial$(\ell, 1-r)$ conditionally on $\{ \xi^{(+)}=\ell\}$.
Comparing
the offspring law \eqref{eq:offGWBer} of the infinite-alleles model with
the offspring law \eqref{eq: mothdepproba} of our model, we find that for every
$i\in[d]$
\begin{equation}\label{eq: equivmod}
      \xi^{(c)} \overset{\mathcal{D}}{=} \boldsymbol{\xi}^{i}(i)
      \qquad \text{and} \qquad
      \xi^{(m)} \overset{\mathcal{D}}{=} \sum_{j\neq i}\boldsymbol{\xi}^{i}(j).
\end{equation}
The two models therefore produce clones and mutants in the same numbers, and
differ only in the alleles that the mutants carry: a novel one in the
infinite-alleles setting, one of the $d-1$ remaining types in ours.

Accordingly, the one-type pair $((T_k, M_{k+1}) : k \in \N_0)$, which records
the size of a clone subfamily in the $k$-th allelic generation together with the
number of mutants it begets, is a Markov chain whose transition probabilities
are governed by $(M_k : k \in \N_0)$, itself a one-type BGW process
\cite[Lemma~1]{Bertoin10}. 

The asymptotic results of~\cite{Bertoin10} concern a sequence of such
populations with initial size $a(n)$ and mutation parameter $r(n)$ satisfying
\begin{equation}\label{eq:Berhyp}
    a(n)  \sim x n  \quad \text{and} \quad r(n)\sim cn^{-1},
    \qquad \text{as } n\rightarrow\infty,
\end{equation}
for constants $x, c > 0$. Under assumption~\ref{hyp:size} and~\eqref{eq:Berhyp},
\cite[Proposition~2]{Bertoin10} establishes the convergence in distribution as $n \to \infty$
\begin{equation}\label{eq:convBerclmut-dCSBP}
    \mathcal{L}\left(\left(  ( n^{-2 }T^{(n)}_k,n^{-1} M^{(n)}_{k+1} )  : k\in\N_0 \right) ;  \mathbb{P}_ {a^{(n)}}^{r(n)} \right) \overset{\mathcal{D}}{\Longrightarrow}
( ( X_{k+1}, cX_{k+1}) : k\in\N_0 ),
\end{equation}
where $X\coloneqq( X_k : k\in\N_0 )$ is a Ji\v{r}ina process
\cite{jivrina1958}, that is, a discrete-time continuous-state branching process
(dCSBP), with initial population size $X_0=c^{-1}x$ and transition
probabilities characterized by
\begin{equation}\label{eq: PM Laplexp 1Jir}
\mathbb{E}( e^{ - \lambda X_{k+1} } \mid X_{k} =x  ) = e^{-x\kappa(\lambda)},
\qquad k\in \N_0, \,  \lambda,x\geq0.
\end{equation}
Here, its branching mechanism is the cumulant of an inverse Gaussian subordinator with
zero drift and L\'evy measure $c^{-1}\nu$, namely
\begin{equation}\label{eq:cum 1-Jirina}
    \kappa(q)= \int_0^\infty (1-e^{-qy})\,c^{-1}\nu(\mathrm{d}y)
    = \frac{c}{\sigma^2}\left(\sqrt{ 1+\frac{2q\sigma^2}{c^2}}-1\right),
    \qquad q\in \R_{+},
\end{equation}
where
\begin{equation} \label{eq: fdpinvgauss}
\nu(\mathrm{d}y)=\frac{c}{\sqrt{2\pi\sigma^2y^3}}
\exp\left(-\frac{c^2y}{2\sigma^2}\right)\mathrm{d}y, \qquad y>0.
\end{equation}
Both $\kappa$ and $\nu$ appear in the description of our own limits, in
Theorems~\ref{thm:limclonemutchain} and~\ref{thm:mainalles}.

\begin{remark}\label{rem:IGparam}
We use the parametrization of the inverse Gaussian law by its mean $m>0$ and shape parameter $\gamma>0$. We have that $W \sim \inverseGauss(m,\gamma)$ if
\[
	\E\left[e^{-qW}\right]=\exp\left\{\frac{\gamma}{m}\left(1 - \sqrt{1+\frac{2m^2q}{\gamma}}\right)\right\}, \qquad q \geq 0 .
\]
Comparing with \eqref{eq:cum 1-Jirina}, for every $m>0$,
\begin{equation}\label{eq:IGid}
	\E\left[e^{-qW}\right] = e^{-m\kappa(q)}
	\qquad \Longleftrightarrow \qquad
	W \sim \inverseGauss\left(\frac{m}{c},\frac{m^2}{\sigma^2}\right).
\end{equation}
Equivalently, if $(\boldsymbol{\tau}_s : s \geq 0)$ denotes the first-passage process of a Brownian motion with variance $\sigma^2$ and drift $-c$, that is, $\boldsymbol{\tau}_s\coloneqq \inf\{t\geq 0 : \sigma B_t - ct = -s\}$, then $\boldsymbol{\tau}$ is an inverse Gaussian subordinator with Laplace exponent $\kappa$ and L\'evy measure $c^{-1}\nu$, and $\boldsymbol{\tau}_m \sim \inverseGauss(m/c, m^2/\sigma^2)$.
\end{remark}

The convergence~\eqref{eq:convBerclmut-dCSBP} records only the total sizes of
the subfamilies in each allelic generation. A tree description is obtained from
the L\'evy--It\^o decomposition, which yields
\cite[Equation~(9)]{Bertoin10}
\begin{equation}\label{eq:atmBer}
    X_1=\sum_{i=1}^\infty a_i,
\end{equation}
where $a_1\geq a_2 \geq \dots$ are the atoms, \textit{ranked in decreasing order} and
repeated according to their multiplicity, of a Poisson random measure on
$(0,\infty)$ with intensity $c^{-1}x\,\nu$, completed by zeros if that measure
is finite. Indexing these atoms by the Ulam-Harris tree $\U\coloneqq \{\varnothing\}\cup\bigcup_{n\geq 1} \N^n$ leads to
the \textit{tree of alleles} and its scaling limit, a tree-indexed dCSBP. We now give the formal definition. 

\begin{definition}\cite[Definition 1]{Bertoin10} \label{def:treeCSBP}
Let $x>0$ and let $\nu$ be a measure on $(0, \infty)$ satisfying $ \int ( 1 \wedge y ) \nu (dy) < \infty $.
A tree-indexed dCSBP with reproduction measure $\nu$ and initial population size $x$ is a process $ ( \mathcal{X}_u  : u \in \U) $ indexed by the Ulam-Harris tree $\U$ and taking values in the nonnegative reals $\mathbb{R}_{+}$, whose distribution is defined inductively by:
\begin{enumerate}
    \item [(i)] $\mathcal{X}_{\varnothing}  = x$ almost surely;
    \item [(ii)] for every $k \in \N_0$, conditionally on the previous levels $ (\mathcal{X}_v : v \in \U, \vert v \vert \leq k) $,
    the sequences $(\mathcal{X}_{uj})_{j \in \mathbb{N}}$ defining the offspring of the vertices $u \in \U$ at level $\vert u \vert = k$ are independent, and each sequence $(\mathcal{X}_{uj})_{j \in \mathbb{N}}$ follows the distribution of the atoms of a Poisson random measure on $(0, \infty)$ with intensity $\mathcal{X}_u \nu$, where these atoms are repeated according to their multiplicity, ranked in decreasing order, and completed by an infinite sequence of zeros if the Poisson measure is finite.
\end{enumerate}
\end{definition}

To describe the prelimit processes in the convergence \eqref{eq:convBerclmut-dCSBP}, the \textit{tree of alleles} $\mathcal{A}\coloneqq( \mathcal{A}_u:u\in\U)$ was constructed recursively in \cite{Bertoin10}, taking advantage of the natural tree structure of $\U$ in order to capture the genealogy of the allelic subfamilies of a BGW process with neutral mutations, along with their respective sizes.
First, define $\mathcal{A}_\varnothing=T_0$. If $u=(u_1,\dots,u_k)$ is a vertex at level $k\in \N$ and $uj=(u_1,\dots,u_k,j)$ denotes its $j$-th child, then $\mathcal{A}_{uj}$ is the size of the $j$-th allelic subpopulation of the $(|u|+1)$-st allelic generation descending from the allelic subfamily indexed by the vertex $u$, ranked in decreasing order of the sizes. In the case of ties, subfamilies are ordered uniformly at random by convention. A further consequence of the general branching property is that the tree of alleles itself enjoys a branching property. To state it, one defines the degree of the tree of alleles $\mathcal{A}$ at a vertex $u\in\U$ as
\[
d_u\coloneqq\max\{ j\geq 1: \mathcal{A}_{uj}>0 \},
\]
with the convention $\max \varnothing =0$, and writes $(d_u\downarrow)$ for the $d_u$-tuple rearranged in decreasing order of the first coordinate, ties being broken uniformly at random.

For each vertex $u$ at level $k$ with $\mathcal{A}_u > 0$, the $d_u$-tuple $((\mathcal{A}_{uj} , d_{uj} ) : 1 \leq j \leq  d_u)$ is distributed as $(T_0 , M_1 )^{(d_u\downarrow)}$ starting from one individual \cite[Lemma 2]{Bertoin10}, and
\begin{equation} \label{obsnivcsbptree}
T_k=\sum_{|u|=k} \mathcal{A}_u \quad \text{and} \quad M_{k+1}=\sum_{|u|=k} d_u, \quad k\in \N_0.
\end{equation}
The convergence in \eqref{eq:convBerclmut-dCSBP} then gives rise to \cite[Theorem 1]{Bertoin10}, which establishes that, under assumption \eqref{eq:Berhyp}, the following joint convergence in the sense of finite-dimensional distributions holds as $n \to \infty$:
\begin{equation} \label{eq:MainThB2}
\mathcal{L}\left( \{ (n^{-2} \mathcal{A}_u, n^{-1} d_u ) : u\in\U \} ; \mathbb{P}_ {a(n)}^{r(n)} \right)
\Longrightarrow
\{(\mathcal{X}_u, c \mathcal{X}_u): u\in\U \}.
\end{equation}
Here $\{ \mathcal{X}_u: u\in \U \}$ is a tree-indexed dCSBP with reproduction measure $c^{-1}\nu$ given by \eqref{eq: fdpinvgauss} and  $\mathcal X_\varnothing$ follows an inverse Gaussian law with mean $x/c$ and shape parameter $x^2/\sigma^2$.

\begin{remark}\label{rem:randomroot}
The root of the tree of alleles carries the whole $0$-th allelic generation, of size $T_0$, and $n^{-2}T_0$ converges to $X_1$, not to $X_0$. This is why the initial law in \eqref{eq:MainThB2} is the law of one Ji\v{r}ina step from $c^{-1}x$ rather than the Dirac mass at $c^{-1}x$. The same phenomenon occurs in the multitype setting, see Theorem~\ref{thm:mainalles}.
\end{remark}

\subsection{Multitype Ji\v{r}ina processes} \label{sec:multiJirina}

The limits appearing in our main results are described in terms of multitype Ji\v{r}ina processes,
that is, multitype discrete-time continuous-state branching processes. We recall
their definition and the constructions we shall use.

Let $\boldsymbol{\zeta}^{(1)},\dots,\boldsymbol{\zeta}^{(d)}$ be independent
$\R_+^d$-valued L\'evy subordinators, and for each $j\in[d]$ let
$\boldsymbol{\kappa}_j : \R_+^d \to \R_+$ denote the Laplace exponent of
$\boldsymbol{\zeta}^{(j)}$, so that
\begin{equation}\label{eq: def kappa_j}
\E\bigl[ e^{-\boldsymbol{\lambda}^\top \boldsymbol{\zeta}^{(j)}_t} \bigr]
= e^{-t\,\boldsymbol{\kappa}_j(\boldsymbol{\lambda})},
\qquad t \geq 0, \ \boldsymbol{\lambda}\in\R_+^d .
\end{equation}
We write $\boldsymbol{\kappa}=(\boldsymbol{\kappa}_1,\dots,\boldsymbol{\kappa}_d)$
for the associated map $\R_+^d\to\R_+^d$, which we call the \textit{branching
mechanism}.

Following \cite{boxmaall-multJirina}, a \textit{multitype Ji\v{r}ina process}
with branching mechanism $\boldsymbol{\kappa}$ and initial state
$\VEC{J}_0=\VEC{x}\in\R_+^d$ is the $\R_+^d$-valued Markov chain
$\VEC{J}=(\VEC{J}_k : k\in\N_0)$ whose transition probabilities are
characterized by
\begin{equation}\label{eq: PM Laplexp dJir}
\E\bigl( e^{ - \boldsymbol{\lambda}^\top \VEC{J}_{k+1} } \mid \VEC{J}_{k} =\VEC{u}  \bigr)
= e^{-\VEC{u}^\top\boldsymbol{\kappa}(\boldsymbol{\lambda})},
\qquad k\in \N_0, \ \VEC{u},\boldsymbol{\lambda}\in \R_+^d.
\end{equation}
Furthermore, iterating \eqref{eq: PM Laplexp dJir} as in \cite[Eq.~1.16]{BlancasRivero}, we characterize the law of the Ji\v{r}ina process as follows. For every $k \in \N$ and $\boldsymbol{\lambda}_1,\dots,\boldsymbol{\lambda}_k\in \R_+^d$,
\begin{equation}\label{eq: rec_Lap_sum_k_gen_dJir}
        \E_{\VEC{x}}\left[\exp\left\{-\sum\limits_{s=1}^{k} \boldsymbol{\lambda}_s^\top \VEC{J}_s\right\}\right]
        =\exp\left\{-\VEC{x}^\top \boldsymbol{\kappa}\bigl(\boldsymbol{\ell}_{k-1}(\boldsymbol{\lambda}_1)\bigr)\right\},
\end{equation}
where $\boldsymbol{\ell}$ is defined by induction as $\boldsymbol{\ell}_0(\boldsymbol{\lambda}_k) = \boldsymbol{\lambda}_k$ and
\begin{equation}\label{eq:defell}
\boldsymbol{\ell}_{s} (\boldsymbol{\lambda}_{k-s} ) \coloneqq \boldsymbol{\lambda}_{k-s} + \boldsymbol{\kappa}\bigl(\boldsymbol{\ell}_{s-1} (\boldsymbol{\lambda}_{k-s+1})\bigr), \qquad s\in [k-1].
\end{equation}

The multitype Ji\v{r}ina process can be constructed coordinatewise from the subordinators. Let $(\boldsymbol{\zeta}^{(j,k)} : j \in [d],\, k \in \N_0)$ be independent copies of $\boldsymbol{\zeta}^{(1)},\dots,\boldsymbol{\zeta}^{(d)}$, one family for each generation $k$. Then
\begin{equation}\label{eq: sub-dJir}
    \VEC{J}_0= \VEC{x}
\qquad \text{and} \qquad
\VEC{J}_{k+1}(i)
=
\sum_{j=1}^d \boldsymbol{\zeta}^{(j,k)}_{\VEC{J}_k(j)}(i),
\qquad k\in \N_0,\, i\in [d],
\end{equation}
defines a process with the transition probabilities \eqref{eq: PM Laplexp dJir}. 

\subsection{Main results}

We consider a sequence of mother-dependent neutral mutations processes
\[
 \left( \VEC{Y}^{(n)} \coloneqq
 (\VEC{Y}_k^{(n)}, k \in\N_0);
 \mathbb{P}_{\VEC{a}^{(n)}}^{r(n)}\right)_{n \in \N},
\]
indexed by $n\in\N$, where $\VEC{Y}^{(n)}$ has mutation probability $r(n)$,
offspring distribution given by~\eqref{eq: mothdepproba}, and initial population
 $\VEC{Y}_0^{(n)}=\VEC{a}^{(n)}$. We work under the following two
asymptotic conditions, which describe a large initial population and rare
mutations.
\begin{enumerate}[label=(H\arabic*)]  
	\item \label{hyp:start} For a fixed vector
	$\VEC{y}\in \R_{+}^d\setminus\{\VEC{0}\}$, the initial population
	$\VEC{a}^{(n)}\in \N^d_0\setminus\{\VEC{0}\}$ satisfies
	\[
		\VEC{a}^{(n)}(i) \sim n\,\VEC{y}(i), \qquad i \in [d],
		\quad \text{ as } n \to \infty ,
	\]
	\item \label{hyp:mutation} For some constant $c > 0$, the mutation
	probability satisfies
	\[
		r(n)  \sim c n^{-1} \quad \text{ as } n \to \infty.
	\]
\end{enumerate}

We now describe the limiting chain. Given $c>0$, the mutation mechanism in
the scaling limit transforms a vector of clone masses $\VEC{x}\in\R_+^d$ into
the corresponding vector of mutant masses $\VEC{y}\in\R_+^d$ through the
relation $\VEC{y}=\VEC{A}\VEC{x}$, where $\VEC{A}$ is the linear operator defined by 
\begin{equation} \label{eq:defAmat}
	(\VEC{A}\VEC{x})(j)= \frac{c}{d-1}\sum_{i\neq j} \VEC{x}(i), \quad j \in [d], 
\end{equation}
for all $\VEC{x} \in \mathbb{R}^d$.
The operator $\VEC{A}$ is invertible, with
\begin{equation}\label{eq:defAmatinv}
	(\VEC{A}^{-1}\VEC{z})(i) = \frac{1}{c}\vert \VEC{z}\vert - \frac{d-1}{c}\VEC{z}(i),
	\qquad i \in [d],
\end{equation}
and it maps $\R_{+}^d$ onto
\begin{equation}\label{eq:defVc}
	\mathcal{V} \coloneqq \VEC{A}(\R_+^d)
	= \left\{ \VEC{z}\in\R_{+}^d : (d-1)\VEC{z}(j) \leq \vert \VEC{z}\vert \text{ for every } j \in [d] \right\}.
\end{equation}

\begin{definition}\label{def:limitchain}
Let $\VEC{y}\in\R_{+}^d\setminus\{\VEC{0}\}$. The \textit{clone-mutant limit chain}
started from $\VEC{y}$ is the $\R_{+}^{2d}$-valued process
$\bigl((\VEC{X}_k,\VEC{Z}_k) : k \in \N\bigr)$ constructed recursively as
follows. Set $\VEC{Z}_0\coloneqq \VEC{y}\in\mathcal{V}$ and, for $k \in \N_0$,
\begin{enumerate}
	\item[(i)] conditionally on $(\VEC{X}_\ell,\VEC{Z}_\ell)_{\ell \leq k}$, the coordinates
	$\VEC{X}_{k+1}(1),\dots,\VEC{X}_{k+1}(d)$ are independent, with
	\begin{equation}\label{eq: IG one step}
		\VEC{X}_{k+1}(j) \sim
		\inverseGauss\left(\frac{\VEC{Z}_{k}(j)}{c},\ \frac{\VEC{Z}_{k}(j)^2}{\sigma^2}\right),
		\qquad j \in [d];
	\end{equation}
	\item[(ii)] $\VEC{Z}_{k+1}\coloneqq \VEC{A} \VEC{X}_{k+1}$, that is,
	\begin{equation}\label{eq:defZ}
		\VEC{Z}_{k+1}(j)= \frac{c}{d-1}\sum_{i\neq j}\VEC{X}_{k+1}(i),
		\qquad j \in [d].
	\end{equation}
\end{enumerate}
\end{definition}

Equivalently, by \eqref{eq:IGid}, the chain $((\VEC{X}_k,\VEC{Z}_k):k \in \N)$ is a homogeneous Markov chain on $\R_{+}^{2d}$ whose transition kernel is characterized by
\begin{equation}\label{eq:transitionlimit}
	\E\left[
    e^{- ( \boldsymbol{\lambda}, \boldsymbol{\theta} )^\top  ( \VEC{X}_{k+1}, \VEC{Z}_{k+1}  ) }
    \, \Bigm\vert \,
    ( \VEC{X}_{k}, \VEC{Z}_{k} ) = (\VEC{u}, \VEC{v} )
    \right]
		=  e^{- \VEC{v}^\top\boldsymbol{\kappa}(\boldsymbol{\lambda}, \boldsymbol{\theta})},
		\qquad  (\boldsymbol{\lambda},\boldsymbol{\theta})\in \R_{+}^{2d}, \quad k\in \N,
\end{equation}
where $\boldsymbol{\kappa}=(\boldsymbol{\kappa}_1,\dots, \boldsymbol{\kappa}_d)$ is given componentwise by
\begin{equation}\label{eq: GMFk}
\boldsymbol{\kappa}_j(\boldsymbol{\lambda}, \boldsymbol{\theta}) \coloneqq
\kappa\left(\boldsymbol{\lambda}(j) + \tfrac{c}{d-1}\sum\limits_{i\neq j }\boldsymbol{\theta}(i)\right)
=
\tfrac{c}{\sigma^2}\left(
   \sqrt{1+ \frac{2\sigma^2}{c^2}\left(\boldsymbol{\lambda}(j) + \tfrac{c}{d-1}\sum\limits_{i\neq j }\boldsymbol{\theta}(i)\right)}
   - 1
\right).
\end{equation}
In particular the transition kernel depends on the current state only through
its mutant coordinate $\VEC{v}$, which is why we parametrize the chain by
$\VEC{Z}_0=\VEC{y}$.

\begin{remark}\label{rem:mutantparam}
It is essential here to parametrize the limit by the mutant masses rather than
by the clone masses. This is one of the points where the finite-allele setting
departs from the infinite-allele one.
The initial condition of the model is the vector $\VEC{y}$ of rescaled initial masses, which plays the role of a mutant
configuration.
The clone state $\VEC{x}$ that would produce it, that is, the
solution of $\VEC{A} \VEC{x}=\VEC{y}$ for the mutation operator $\VEC{A}$
of~\eqref{eq:defAmat}, has nonnegative entries if and only if
$(d-1)\VEC{y}(j)\leq \vert\VEC{y}\vert$ for every $j\in[d]$. 
\end{remark}

We can now state our first main theorem.

\begin{theorem}\label{thm:limclonemutchain}
Assume \ref{hyp:size}, \ref{hyp:start} and \ref{hyp:mutation}, and let
$\bigl((\VEC{X}_k,\VEC{Z}_k):k\in\N\bigr)$ be the clone-mutant limit chain
started from $\VEC{y}$, as in Definition~\ref{def:limitchain}. Then, as $n \to \infty$,
\begin{equation} \label{eq:conv-clonemut}
	\mathcal{L}\left(\left(  \bigl( n^{-2 }\VEC{T}^{(n)}_k,\ n^{-1} \VEC{M}^{(n)}_{k+1} \bigr)  : k\in\N_0 \right) ;  \mathbb{P}_ {\VEC{a}^{(n)}}^{r(n)} \right)
	\overset{\mathcal{D}}{\Longrightarrow}
\bigl( ( \VEC{X}_{k+1}, \VEC{Z}_{k+1}) : k\in\N_0 \bigr).
\end{equation}
Moreover, $(\VEC{X}_k : k \in N)$ is a multitype Ji\v{r}ina process with
branching mechanism $\boldsymbol{\kappa}=(\boldsymbol{\kappa}_1,\dots,
\boldsymbol{\kappa}_d)$ given by
\begin{equation}\label{eq: cum_type_j_Jir}
	\boldsymbol{\kappa}_i(\boldsymbol{\lambda})
	= \frac{c}{d-1}\sum_{j\neq i}\kappa(\boldsymbol{\lambda}(j)),
    \qquad \boldsymbol{\lambda}=(\boldsymbol{\lambda}(1),\dots ,\boldsymbol{\lambda}(d))\in \R_+^d ,
\end{equation}
where $\kappa$ is given by~\eqref{eq:cum 1-Jirina}.
\end{theorem}

The proof of the theorem above is in Section~\ref{sec:proofmain1}. 

\begin{remark}\label{rem:totalmass}
The total mass of the limit reproduces the infinite-alleles picture. Summing
\eqref{eq:defZ} over $j$, each $\VEC{X}_k(i)$ being counted once for every
$j \neq i$, gives
\[
	\vert \VEC{Z}_{k}\vert = c \vert \VEC{X}_{k}\vert, \qquad k \in \N .
\]
Moreover, taking $\boldsymbol{\theta}=\VEC{0}$ and
$\boldsymbol{\lambda}=\lambda\VEC{1}$ in \eqref{eq:transitionlimit}, and noting
that $\boldsymbol{\kappa}_i(\lambda\VEC{1}) = c\,\kappa(\lambda)$ for every
$i \in [d]$ by \eqref{eq: cum_type_j_Jir}, we obtain
\[
	\E\bigl[e^{-\lambda\vert\VEC{X}_{k+1}\vert} \bigm\vert \VEC{X}_{k}=\VEC{u}\bigr]
	= e^{-c\vert\VEC{u}\vert\,\kappa(\lambda)} .
\]
Thus $(\vert\VEC{X}_k\vert : k \in \N)$ is a one-type Ji\v{r}ina process with
cumulant $\kappa$, issued from
$\vert\VEC{x}\vert = \vert\VEC{y}\vert/c$, and $\vert\VEC{Z}_k\vert$ plays the
role of the mutant mass $cX_k$ of \eqref{eq:convBerclmut-dCSBP}. The total mass
at each allelic generation therefore behaves as in the infinite-alleles case. The novelty here is the partition of that mass among the $d$ types.
\end{remark}

We now lift Theorem~\ref{thm:limclonemutchain} to the genealogical level by
endowing both the clone-mutant chain $((\VEC{T}_{k},\VEC{M}_{k+1}) : k\in\N_0)$
and its limit with a tree structure. We  index the genealogy by the Ulam--Harris tree $\U$, where each $u$ identifies an allelic subfamily. We further introduce two mark functions, 
$\mathcal{B}:\mathcal{U} \to \N_0$ and $\mathcal{C}:\mathcal{U} \to [d]$, where $\mathcal{C}_u$ denotes the type of the allelic subfamily characterized by $u$, while $\mathcal{B}_u$ denotes its size. Subsection~\ref{sec:alleleTree} gives a careful construction of the $d$-type tree of alleles
$(\mathscr{A}_u : u\in\U)$ associated with the clone-mutant Markov chain, which is the
finite-allele analog of the tree of alleles introduced in~\cite{Bertoin10}. Each marked vertex
$\mathscr{A}_u \coloneqq (\mathcal{B}_u,\mathcal{C}_u)$, encoding one allelic
subfamily, generates an offspring vector $\VEC{d}_u\in\N_0^d$, 
where $\VEC{d}_u(j)$ is the number of type-$j$ mutants begotten by the
subfamily $u$.
More precisely, a given vertex $u$ has $|\VEC{d}_u|$ children. For each $\ell \in \{1, \dots,|\VEC{d}_u|\}$, if 
$\sum_{i=1}^{j-1}\VEC{d}_u(i) < \ell\leq \sum_{i=1}^{j}\VEC{d}_u(i)$, the vertex $u\ell$  corresponds to the $\left(\ell-\sum_{i=1}^{j-1}\VEC{d}_u(i)\right)$-th largest subfamily of type $j$ among the
children of $u$. Thus, $\mathcal{B}_{u\ell}>0$ and $\mathcal{C}_{u\ell}=j$.

We now use the subordinator construction introduced in
Subsection~\ref{sec:multiJirina} to define the limiting genealogical object. While~\eqref{eq: sub-dJir} records only the total
mass produced in each type, the L\'evy--It\^o decomposition yields a
genealogical refinement of this construction: for a subfamily of type $i$ and mass $x$, the masses of its type-$j$ children, for $j\neq i$, are
the ranked jumps of the $j$-th coordinate of
$\boldsymbol{\zeta}^{(i)}_x$. Equivalently, they are the atoms of a Poisson
random measure on $(0,\infty)$ with intensity
$x(d-1)^{-1}\nu$. Summing these atoms over all vertices of a given type at each
level recovers the corresponding multitype Ji\v{r}ina process. This motivates
the following tree-indexed construction.
\begin{definition} \label{def:multi-treeCSBP}
Let $i \in [d]$ and let $\varrho$ be a probability measure on $\R_{+}$. A
\textit{tree-indexed multitype Ji\v{r}ina process} with branching mechanism
$\boldsymbol{\kappa}$ as in \eqref{eq: cum_type_j_Jir}, root type $i$ and root
law $\varrho$, is a process $(\mathscr{Y}_u \coloneqq(
\mathcal{Y}_u,\mathcal{C}_u ): u \in \U)$ with values in $\R_{+}\times [d]$
whose distribution is defined inductively by:
\begin{enumerate}
    \item [(i)] $\mathcal{C}_{\varnothing}=i$ and $\mathcal{Y}_{\varnothing} \sim \varrho$;
    \item [(ii)] for every $k \in \N_0$, conditionally on the previous levels $\left(\mathscr{Y}_v : v \in \U, \vert v \vert \leq k \right)$,
    the families $\bigl(\mathcal{Y}_{u\ell} :  \ell \in \N\bigr)$ associated with the distinct vertices $u \in \U$ at level $\vert u \vert = k$ are independent;
    \item [(iii)] for each such $u$, the sequences $(\mathcal{Y}_{u\ell} : \ell\in\N)$ are independent; and for each $j\neq \mathcal{C}_u$ the sequence $(\mathcal{Y}_{u\ell}: \mathcal{C}_{u\ell}=j)_{\ell\in\N}$ is the decreasing rearrangement of the atoms of a Poisson random measure on $(0,\infty)$ with intensity $\frac{1}{d-1}\mathcal{Y}_u\,\nu$, 
    where $\nu$ is the Lévy measure associated to the branching mechanism $ \boldsymbol{\kappa} $, 
    repeated according to their multiplicity and completed by an infinite sequence of zeros if that measure is finite. Finally $(\mathcal{Y}_{uk}: \mathcal{C}_{uk}=\mathcal{C}_u)=\emptyset$.
\end{enumerate}
\end{definition}

By construction, for every $k \in \N_0$ and $j \in [d]$, the total mass of the
type-$j$ vertices at level $k$ of a tree-indexed multitype Ji\v{r}ina process
started from $\varrho$ has the law of $\VEC{X}_{k+1}(j)$ in
Definition~\ref{def:limitchain}, provided $\varrho$ is the law of
$\VEC{X}_1(i)$. 
This is the limiting counterpart of the identities
\eqref{eq:alltreeCloMutch} below.

We are now in a position to state our second theorem.

\begin{theorem} \label{thm:mainalles}
	Fix $i\in[d]$ and $y>0$. Assume \ref{hyp:size}, \ref{hyp:mutation}, and, in place of \ref{hyp:start},
	that the initial population satisfies
	\begin{equation} \label{eq: allestart1}
		\VEC{a}^{(n)} \sim n y\VEC{e}_i,  \quad \text{ as } n \to \infty.
	\end{equation}
	Then the following joint convergence holds in the sense of finite-dimensional distributions:
 	\[
		\mathcal{L}\left(
		\left( ( n^{-2}\mathcal{B}^{(n)}_u, \mathcal{C}^{(n)}_u, n^{-1}\boldsymbol{d}^{(n)}_u) : u \in \U\right);\,
		\mathbb{P}_{\VEC{a}^{(n)}}^{r(n)}
		\right) \xLongrightarrow{\mathcal{D}}
		\left( (\mathcal{Y}_u, \mathcal{C}_u, \boldsymbol{\mathcal{Z}}_u) : u \in \U \right), \quad \text{as } n\to \infty,
	\]
	where $\left(\mathscr{Y}_u\coloneqq (\mathcal{Y}_u,\mathcal{C}_u) : u \in \U\right)$ is the tree-indexed multitype Ji\v{r}ina process of Definition~\ref{def:multi-treeCSBP} with branching mechanism $\boldsymbol{\kappa}$ given in~\eqref{eq: cum_type_j_Jir}, root type $i$ and root law $\varrho$ corresponding to $\inverseGauss\left( \frac{y}{c},\frac{y^2}{\sigma^2}\right)$, and where $\boldsymbol{\mathcal{Z}}_u
	\coloneqq \frac{c}{d-1}\, \mathcal{Y}_u \, \bare{\mathcal{C}_u}$, $u\in \U$, that is
	\[
    \boldsymbol{\mathcal{Z}}_u(j)=\frac{c}{d-1}\mathcal{Y}_u \quad  \text{ for } \quad j \neq \mathcal{C}_u,  \quad \text{ and } \quad\boldsymbol{\mathcal{Z}}_u(\mathcal{C}_u)=0 .
	\]
\end{theorem}

The theorem admits the following interpretation. In the limit, a clone subfamily
represented by a vertex $u$ of type $\mathcal{C}_u$ and mass $\mathcal{Y}_u$
begets, in the next allelic generation and independently for each type $j \neq
\mathcal{C}_u$, an infinite ranked sequence of clone subfamilies of type $j$
whose masses are the atoms of a Poisson random measure with intensity
$\frac{1}{d-1}\mathcal{Y}_u\nu$. The total mass so created in type $j$ is
\[
	\sum_{\ell :\, \mathcal{C}_{u\ell}=j}\mathcal{Y}_{u\ell} \sim
	\inverseGauss\left(\frac{\mathcal{Y}_u}{d-1}, \frac{c^2\,\mathcal{Y}_u^{2}}{\sigma^2 (d-1)^2}\right),
\]
in accordance with \eqref{eq: IG one step}, whereas the \textit{number} of mutants
founding these subfamilies is, after the $n^{-1}$ rescaling, the deterministic quantity
$\boldsymbol{\mathcal{Z}}_u(j)=\frac{c}{d-1}\mathcal{Y}_u$. This phenomenon is due to the scaling, which absorbs the fluctuations of the mutant counts,  as in
\eqref{eq:convBerclmut-dCSBP}.

Table~\ref{table:compresult} summarizes the correspondence between our results and their infinite-allele counterparts.

\begin{table}[htbp]
\centering
\caption{Comparison of the finite- and infinite-allele settings.}
\label{table:compresult}
\begin{tabular}{ |p{5cm} | l | l| } 
\hline
Description  & Finite alleles   &  Infinite alleles \cite{Bertoin10}    \\[.35ex]
\hline
Clone-mutant chain transition probabilities  &  Lemma \ref{le:TyMMarkov} & Lemma 1 \\
\hline
Branching property of the allele tree  &  Lemma \ref{lem: transalletree} & Lemma 2 \\
\hline
First-step distribution of the clone-mutant Markov chain  &  Proposition \ref{Prop: pgf T_0Mvec} & Proposition 1  \\
\hline
First and second moments of the number of mutants  &  Corollary~\ref{cor:ext} & Corollary 1 \\
\hline
Equality in distribution for the random walk representation  &   Lemma \ref{lem: simL3Ber} & Lemma 3 and Corollary 2 \\
\hline
Convergence of the rescaled clone-mutant Markov chain  &  Theorem \ref{thm:limclonemutchain} & Proposition 2   \\
\hline
Convergence of the rescaled allele tree  &  Theorem \ref{thm:mainalles} & Theorem 1  \\
\hline
Convergence of the rescaled first-step  &  Proposition \ref{prop:conv-step1} & Lemma 4    \\
\hline
\end{tabular}
\end{table}

\section{Distribution of a typical clone subfamily}\label{sec:bases}

In this section, we characterize the transition probabilities of the clone-mutant Markov chain associated with a mother-dependent neutral mutation model. The following results enhance Lemma~\ref{le:TyMMarkov}. 
Our main tool is generating functions, through which we express the joint distribution $(\VEC{T}_0, \VEC{M}_1)$ in terms of the offspring distribution of the underlying $d$-type BGW process.

Let $\{  (\VEC{Y}_k, k \in \mathbb{N}_0); \mathbb{P}_{\VEC{e}_i}^{r}\}$ be a $d$-type BGW process
representing a population starting from a single individual of type $i$, with mother-dependent neutral mutations. Then
\[
	\VEC{T}_0= \VEC{T}_0(i) \VEC{e}_i, \quad \VEC{M}_0=\VEC{e}_i, \quad \text{ and } \quad \VEC{M}_1(i)= 0,
\]
as stated in Remark~\ref{rem:mutdiftypmoth}. 
Let $\varphi_{i}$ denote the generating function of $(\VEC{T}_0(i), \VEC{M}_1)$ under $\mathbb{P}_{\VEC{e}_i}^{r}$:
\begin{equation} \label{eq:MGFone}
	\varphi_{i}(x,\VEC{\bar{y}}^i)
	\coloneqq
	\mathbb{E}_{\VEC{e}_i}^r\left[ x^{ \VEC{T}_0(i) } \prod_{j\neq i}   \VEC{y}(j)^{ \VEC{M}_1(j) } \right], \qquad x \in [0,1],\, \VEC{y} \in [0,1]^d.
\end{equation}

More generally, for $\{  (\VEC{Y}_k, k \in \mathbb{N}_0); \mathbb{P}_{\VEC{a}}^{r}\}$ a $d$-type BGW process starting from $\VEC{a}\in \N_0^d\setminus\{\VEC{0}\}$, the generating function of $(\VEC{T}_0, \VEC{M}_1)$ is denoted by $\varphi_{\boldsymbol{a}}$, that is,
\begin{equation}\label{eq: MGFgen}
\varphi_{\boldsymbol{a}}( \VEC{x}, \VEC{y})
	\coloneqq
	\mathbb{E}_{\VEC{a}}^r\left[ \prod_{j=1}^d \VEC{x}(j)^{ \VEC{T}_0(j) }  \VEC{y}(j)^{ \VEC{M}_1(j) } \right] ,
	\qquad \VEC{a}\in \N_0^d\setminus\{\VEC{0}\},\,  \VEC{x} \in [0,1]^d,\, \VEC{y} \in [0,1]^d.
\end{equation}

Let ${g}_{i}$ denote the generating function of the multivariate offspring distribution $\boldsymbol{\mu}_i$ for individuals of type $i$:
\begin{equation}\label{eq: g_i}
	{g}_{i}(\VEC{s})
	= g_{i}( \VEC{s}(i) , \VEC{\bar s}^i)\coloneqq
	\mathbb{E}_{\VEC{e}_i}^r\left[ \VEC{s}(i)^{ \boldsymbol{\xi}^{i}(i) }  \prod_{j\neq i}  \VEC{s}(j)^{   \boldsymbol{\xi}^{i}(j)  }  \right]
	= \sum_{\VEC{v} \in \mathbb{N}_0^d} \VEC{s}(i)^{ \VEC{v}(i) } \prod_{j\neq i}  \VEC{s}(j)^{ \VEC{v}(j)}  \boldsymbol{\mu}_i (\VEC{v}),
\end{equation}
$\VEC{s} \in [0,1]^d$. We use the notation in the second term to emphasize the roles of clones and mutants: whereas the first argument tracks clone offspring (all individuals of type $i$), the second tracks mutants (all individuals of types $j \neq i$). Here we use the notation introduced in~\eqref{eq:notationErase}.

For the process $\VEC{Y} \coloneqq \left(\VEC{Y}_k : k \in \mathbb{N}_0\right)$ with law $\mathbb{P}_{\VEC{e}_i}^{r}$, the function $g_i$ is the generating function of $\VEC{Y}_1$.
Now, let us assume an initial population $\VEC{Y}_0=\VEC{a} \in \mathbb{N}^d_0\setminus\{\VEC{0}\}$. In this case, the generating function of $\VEC{Y}_1$ satisfies, for each $\VEC{s} \in [0,1]^d$,
\begin{align}
	g_{\VEC{a}}(\VEC{s})
		&\coloneqq \mathbb{E}_{\VEC{a}}^r\left[\prod_{j=1}^d \VEC{s}(j)^{ \VEC{Y}_1(j)} \right]
		=\mathbb{E}_{\VEC{a}}^r \left[ \prod_{j=1}^d \VEC{s}(j)^{\sum\limits_{i=1}^d \sum\limits_{\ell=1}^{\VEC{a}(i)} \boldsymbol{\xi}_{1,\ell}^{i}(j) } \right] \nonumber \\
		&=  \prod_{i=1}^d \left(\mathbb{E}_{\VEC{e}_i}^r \left[ \VEC{s}(i)^{   \boldsymbol{\xi}^{i}(i)  }\prod_{j\neq i} \VEC{s}(j)^{ \boldsymbol{\xi}^{i}(j) } \right]\right)^{\VEC{a}(i)} =  \prod_{i=1}^d \left(g_i( \VEC{s}(i) , \VEC{\bar s}^i) \right) ^{\VEC{a}(i)}, \label{eq:eq1MGF}
\end{align}
where the second line follows from the branching property, see \eqref{eq: product measure}. We also define
\[
	\VEC{g}(\VEC{s})  \coloneqq \left(g_{1}(\VEC{s}),\dots,g_{d}(\VEC{s}) \right) = \left(g_{1}(\VEC{s}(1) , \VEC{\bar s}^1),\dots,g_{d}(\VEC{s}(d) , \VEC{\bar s}^d)\right),
\]
so that $g_{\VEC{a}}(\VEC{s})=( \VEC{g}(\VEC{s}) )^{\VEC{a}}$, with the product notation of Subsection~\ref{sec:notation}.

The following proposition describes the transition probabilities of the clone-mutant Markov chain under two possible initial conditions: it first considers an initial population consisting of individuals of the same type, and then the general case $\VEC{a} \in \mathbb{N}^d_0\setminus\{\VEC{0}\}$.

\begin{proposition}\label{Prop: pgf T_0Mvec}
Under assumption \ref{hyp:size}, the distribution of the random vector $(\VEC{T}_0, \VEC{M}_1)$ satisfies the following.
\begin{enumerate}
	\item  [(i)]  For each $i\in [d]$, the generating function of $(\VEC{T}_0(i), \VEC{M}_1)$ defined in \eqref{eq:MGFone} satisfies the equation
		\begin{equation}\label{eq: pgfTM_i}
			\varphi_i( x, \VEC{\bar{y}}^i) =  x\, g_i(\varphi_i(x, \VEC{\bar{y}}^i), \VEC{\bar{y}}^i), \qquad   x\in [0,1], \ \VEC{y}\in [0,1]^d,
		\end{equation}
        and, for every fixed $\VEC{\bar{y}}^i\in(0,1]^d$, it is the unique power series in $x$ with these properties.
        Furthermore, for any $\VEC{a}(i)\in\N$, the distribution of $(\VEC{T}_0, \VEC{M}_1)$ under $\mathbb{P}_{\VEC{a}(i)\mathbf{e}_i}^r$ is given by
		\[
		\mathbb{P}_{\VEC{a}(i)\VEC{e}_i}^r(\VEC{T}_0=m\VEC{e}_i, \VEC{M}_1=\boldsymbol{\ell})
		= \frac{\VEC{a}(i)}{m}\boldsymbol{\mu}_i^{*m}(\boldsymbol{\ell} + (m-\VEC{a}(i))\VEC{e}_i),
        \qquad  m \geq \VEC{a}(i),\ \boldsymbol{\ell} \in \bar{\N}_0^i,
		\]
		where $\boldsymbol{\mu}_i^{*m}$ denotes the $m$-th convolution power of $\boldsymbol\mu_i$.

	\item [(ii)] For any $\mathbf{a}\in \mathbb{N}^d_0 \setminus\{ \VEC{0}\}$, the generating function of $(\VEC{T}_0, \VEC{M}_1)$ under $\P^r_{\VEC{a}}$, defined in~\eqref{eq: MGFgen}, satisfies
	\begin{equation}\label{eq: pgfgenstart}
	    \varphi_{\VEC{a}}( \VEC{x}, \VEC{y}) =(\boldsymbol{\varphi}( \VEC{x}, \VEC{y}))^{\VEC{a}}, \qquad  (\VEC{x},\VEC{y}) \in [0,1]^d \times [0,1]^d; 
	\end{equation}
where
\[
	\boldsymbol{\varphi}( \VEC{x}, \VEC{y})
	\coloneqq
	\left( \varphi_{1}(\VEC{x}(1),\VEC{\bar{y}}^1) ,\dots, \varphi_{d}(\VEC{x}(d),\VEC{\bar{y}}^d) \right).
\]
Furthermore, for every $\VEC{m}\in\N^d$ with $\VEC{m}\boldsymbol{\geq}\VEC{a}$ and every $\boldsymbol{\ell}\in\N_0^d$,
\[
	\mathbb{P}_{\VEC{a}}^r\left(\VEC{T}_0=\VEC{m}, \VEC{M}_1 = \boldsymbol{\ell}\right)
    = \sum\limits_{(\VEC{w}_1,\dots,\VEC{w}_d) \in W_{\boldsymbol{\ell}}} \ \prod_{i=1}^d \frac{\VEC{a}(i)}{\VEC{m}(i)}\, \boldsymbol{\mu}_i^{*\VEC{m}(i)}(\VEC{w}_i + (\VEC{m}(i)-\VEC{a}(i) )\VEC{e}_i),
\]
		where $ W_{\boldsymbol{\ell}} \coloneqq\{ (\VEC{w}_1,\dots,\VEC{w}_d) \in \bar{\N}_0^1 \times \dots \times \bar{\N}_0^d : \sum_{i=1}^d \VEC{w}_i = \boldsymbol{\ell} \}$.
	\end{enumerate}
\end{proposition}

As a corollary, the total number of mutant children $|\VEC{M}_1|$ is critical.

\begin{corollary} \label{cor:ext}
Let $r \in (0,1)$ and
assume that~\ref{hyp:size} holds.
	\begin{itemize}
\item [(i)] For every $\mathbf{a}\in \mathbb{N}^d_0\setminus \{\VEC{0}\} $ and every $i\in [d]$,
\[
 \mathbb{E}_{\VEC{a}}^r\left[\VEC{T}_0(i)\right]=\frac{\VEC{a}(i)}{r}
 \qquad \text{and} \qquad
 \mathbb{E}_{\VEC{a}}^r\left[\VEC{M}_1(i)\right] = \sum\limits_{j\neq i} \frac{\VEC{a}(j)}{d-1}.
\]
\item  [(ii)] For every $i \in [d]$, the total number of mutants in the first allelic generation satisfies $\E_{\VEC{e}_i}^r[\vert \VEC{M}_1\vert]=1$, and $\E_{\VEC{e}_i}^r[\vert \VEC{M}_1\vert^2]<\infty$.
\end{itemize}
\end{corollary}

We postpone the proofs of the above results to Appendix~\ref{app:proofsim}, as they closely follow the arguments in \cite[Proposition 1 and Corollary 1]{Bertoin10}. The differences in the statements arise from the fact that we restrict our attention to critical distributions.

\section{Two representations of the allelic subfamilies} \label{sec: representations}

Consider a $d$-type BGW process with mother-dependent neutral mutations $ \{ \VEC{Y}\coloneqq(\VEC{Y}_k, k \in \mathbb{N}_0); \mathbb{P}_{\VEC{e}_i}^{r}\} $ starting with a single individual of type $i$. Associated with $\VEC{Y}$ are the clone-mutant Markov chain $( (\mathbf{T}_k, \mathbf{M}_{k+1}) : k \in \mathbb{N}_0 ) $ and 
 the marked planar tree $\mathcal{T}_d $ with colored vertices and marked edges, where the colors indicate the types of the individuals and the marks indicate mutation, associated with the probability measure $\mathbb{P}_{\VEC{e}_i}^{r}$.

The goal of this section is to show that the transition probabilities of the clone-mutant Markov chain admit an interpretation in terms of both random walks and tree structures. To achieve this, we label individuals recursively according to their allelic generation using a breadth-first search algorithm within each allelic subfamily. Individuals in allelic generation $0$ are labeled first, then those in generation $1$, and so forth. See Figure \ref{fig:treerwalle}, which illustrates this construction.
 
 \begin{figure}[htbp]
    \centering
    \includegraphics[width=1.0\linewidth]{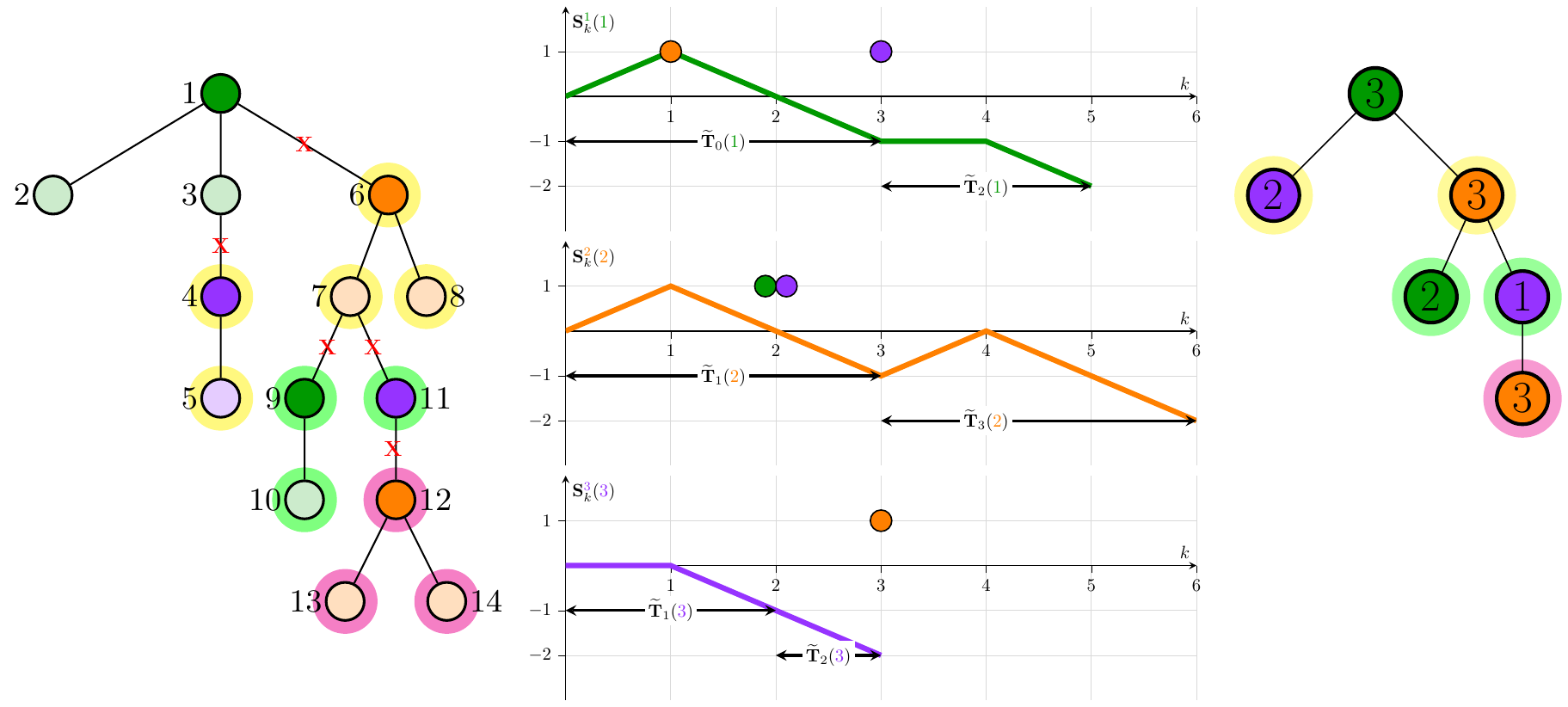}
    \caption{Left: A realization of the mother-dependent neutral mutations model with $d = 3$ types. The types $1$, $2$ and $3$ are represented in green, orange, and purple, respectively. Individuals are labeled according to a breadth-first search algorithm within each allelic subfamily. Individuals in the first, second, and third allelic generations are highlighted in yellow, green, and pink, respectively. Center: The random walk associated with each type.
    Right: Multitype allele tree associated with the colored genealogical tree on the left, defined in Subsection \ref{sec:alleleTree}.}
    \label{fig:treerwalle}
\end{figure}

 \subsection{From random walks to the clone-mutant Markov chain}\label{subsec: rw}

For each $i\in [d]$ we define a $d$-dimensional random walk
\[
    \VEC{S}_{k}^{(i)}=(\VEC{S}_{k}^{(i)}(1), \dots, \VEC{S}_{k}^{(i)}(d) ), \qquad k \in \N_0,
\]
by $\VEC{S}^{(i)}_{0}\coloneqq \VEC{0}$ and
\begin{align}\label{eq: drw}
	\VEC{S}^{(i)}_{k+1}
    \coloneqq
    \VEC{S}^{(i)}_{k} + \left(\boldsymbol{\xi}_{k+1 }^{i}(i) -1\right)\VEC{e}_i + \sum_{j\neq i} \boldsymbol{\xi}_{k+1}^{i}(j) \VEC{e}_j , \qquad k\in \N_0,
\end{align}
where $(\boldsymbol{\xi}_{k}^{i}:k \in \N)$ is a sequence of i.i.d. random vectors with distribution $\boldsymbol{\mu}_i$; see Figure~\ref{fig:treerwalle}. In other words,
\begin{equation*}
\VEC{S}^{(i)}_k (j) =
    \left\{
    \begin{array}{ll}
   \sum\limits_{\ell = 1}^{k} \left(  \boldsymbol{\xi}_{\ell }^{i}(i) -1\right),   &\qquad  \text{ if } j=i, \\[1ex]
    \sum\limits_{\ell=1}^k\boldsymbol{\xi}_{\ell}^{i}(j),  &\qquad  \text{ if } j\neq i.
    \end{array}
    \right.
\end{equation*}
\begin{remark}\label{rem:indrw}
  The $d$ walks $\VEC{S}^{(1)},\dots,\VEC{S}^{(d)}$ are built from disjoint families of offspring variables and are therefore independent.
\end{remark}
Recalling \eqref{eq: equivmod}, one may compare the walk \eqref{eq: drw} with the one defined in \cite[Section 2.3]{Bertoin10} for a one-type BGW process with neutral mutations and infinitely many alleles.

Since we want to analyze the sizes of the allelic subfamilies, we use the well-known embedding of a BGW process in a random walk (see, e.g.~\cite{Dwass69} and~\cite{Loic2016codmult} for a multidimensional setting). To do so, we introduce for every $\ell\in\N_0$ and $i\in[d]$ the hitting time
\begin{equation} \label{eq:tau}
\boldsymbol{\tau}_{\ell}(i) \coloneqq \inf\{ k\in\N_0: \VEC{S}_{k}^{(i)}(i)=-\ell \}.
\end{equation}
Note that assumption \ref{hyp:size} ensures that it does not drift to $+\infty$, and hence the passage times $\boldsymbol{\tau}_{\ell}(i)$ are a.s. finite.

Based on the above construction, we set $\widetilde{\VEC{M}}_{0}\coloneqq\VEC{a}$,
\[
\VEC{\widetilde{T} }_0 \coloneqq (\boldsymbol{\tau}_{\VEC{a}(1)}(1), \dots, \boldsymbol{\tau}_{\VEC{a}(d)}(d) )
\quad \text{and} \quad
\widetilde{\VEC{M}}_{1}(i) \coloneqq
  \sum\limits_{j\neq i} \VEC{S}_{ \VEC{\widetilde{T} }_{0}(j) }^{(j)}(i)=\sum\limits_{j\neq i} \sum_{\ell=1}^{\VEC{\widetilde{T} }_{0}(j)}\boldsymbol{\xi}_{\ell}^{j}(i), \quad i\in[d],
\]
and, for every $k\in\N$ and $i\in[d]$, recursively
\begin{equation}\label{eq: cmintermsofrw}
\sum_{\ell=0}^{k}\VEC{\widetilde{T} }_{\ell}(i)
\coloneqq
\boldsymbol{\tau}_{\sum_{\ell=0}^{k}  \widetilde{\VEC{M}}_{\ell}(i)}(i)
\qquad \text{ and } \qquad
\sum_{\ell=1}^{k}\VEC{\widetilde{M} }_{\ell}(i)
\coloneqq  \sum\limits_{j\neq i} \VEC{S}_{ \sum_{\ell=0}^{k-1}\VEC{\widetilde{T} }_{\ell}(j)}^{(j)}(i).
\end{equation}

\begin{remark}\label{re: eqdist}
Since for every $i\in[d]$ the random vectors $(\boldsymbol{\xi}^i_{k})_{k \in \N}$ are i.i.d., the strong Markov property of $\VEC{S}^{(i)}$ at the times $\boldsymbol{\tau}_\ell(i)$ gives, for every $m\in\N_0$,
\[
    \boldsymbol{\tau}_m(i)\overset{\mathcal{D}}{=} \widetilde{\boldsymbol{\tau}}^1_1(i)+\dots +\widetilde{\boldsymbol{\tau}}^m_1(i),
    \qquad
    \VEC{S}_{\boldsymbol{\tau}_m(i)}^{(i)}(j)
    \overset{\mathcal{D}}{=}
    \widetilde{\VEC{S}}_{\widetilde{\boldsymbol{\tau}}_1}^{(i,1)}(j)+\dots+
    \widetilde{\VEC{S}}_{\widetilde{\boldsymbol{\tau}}_1}^{(i,m)}(j),
\]
jointly, where $\{\widetilde{\boldsymbol{\tau}}^\ell_1(i)\}_{\ell=1}^m$ and $\{\widetilde{\VEC{S}}_{\widetilde{\boldsymbol{\tau}}_1}^{(i,\ell)}(j)\}_{\ell=1}^m$ are independent copies of $\boldsymbol{\tau}_1(i)$ and of $\VEC{S}_{\boldsymbol{\tau}_1(i)}^{(i)}(j)$. Hence \eqref{eq: cmintermsofrw} implies
\[
    \VEC{\widetilde{T} }_{k}(i)\overset{\mathcal{D}}{=}\boldsymbol{\tau}_{  \widetilde{\VEC{M}}_{k}(i)}(i)
    \qquad \text{and} \qquad
    \widetilde{\VEC{M}}_{k+1}(i)\overset{\mathcal{D}}{=}
    \sum\limits_{j\neq i} \VEC{S}_{ \VEC{\widetilde{T} }_{k}(j) }^{(j)}(i) . 
 \]
\end{remark}

The next lemma relates the total size of an allelic subfamily to the first hitting times of the random walks just defined.

\begin{lemma} \label{lem: simL3Ber}
For every $\VEC{a} \in \mathbb{N}_0^d\setminus\{ \VEC{0}\}$, under $\P_{\VEC{a}}^r$ the pairs of random vectors
\[
    \left( \mathbf{T}_0,\mathbf{M}_{1}   \right)
    \qquad \text{and} \qquad
    \left( \widetilde{\mathbf{T}}_0,\widetilde{\mathbf{M}}_{1}   \right)
\]
have the same distribution. Moreover, $( (\widetilde{\mathbf{T}}_k, \widetilde{\mathbf{M}}_{k+1}): k \in \mathbb{N}_0 )$ is a Markov chain with the same distribution as
$\left( (\mathbf{T}_k,\mathbf{M}_{k+1}): k \in \mathbb{N}_0   \right)$.
\end{lemma}

\begin{proof}\label{proof: simL3Ber}
First, assume that the starting population consists of one individual of type $i$, i.e. $\widetilde{\VEC{M}}_{0}=\VEC{e}_i$. Then $\VEC{\widetilde{T} }_0 = \boldsymbol{\tau}_1(i)\VEC{e}_i$ and
\[
\widetilde{\VEC{M}}_{1}(i)=0,
\qquad \widetilde{\VEC{M}}_{1}(j) =
 \VEC{S}_{ \boldsymbol{\tau}_{1}(i)  }^{(i)}(j)
  = \sum_{\ell=1}^{\boldsymbol{\tau}_{1}(i)}\boldsymbol{\xi}_{\ell}^{i}(j), \qquad j\neq i.
\]
Let $\widetilde{\varphi}_{i}$ be the generating function of $( \widetilde{\VEC{T}}_0(i),\widetilde{\VEC{M}}_{1} ) $ under $\P_{\VEC{e}_i}^r$. For $x\in [0,1]$ and $\VEC{y}\in [0,1]^d$,
\begin{equation*}
\begin{split}
	\widetilde{\varphi}_{i}( x, \VEC{\bar{y}}^i)
	& = \mathbb{E}_{\VEC{e}_i}^r\left[
    x^{ \boldsymbol{\tau}_1(i) } \prod_{j\neq i}  \VEC{y}(j)^{ \VEC{S}_{\boldsymbol{\tau}_1(i)}^{(i)}(j) }
    \right]  \\
	& = \sum\limits_{m=0}^\infty \sum\limits_{\boldsymbol{\ell} \in \bar{\N}_0^i}   \mathbb{P}_{\VEC{e}_i}^r\left[
    \boldsymbol{\xi}^{i}=\boldsymbol{\ell}+m\VEC{e}_i\right]
	 \mathbb{E}^r
     \left[
     x^{ \boldsymbol{\tau}_m(i) +1 } \prod_{j\neq i}  \VEC{y}(j)^{ \VEC{S}_{\boldsymbol{\tau}_m(i)}^{(i)}(j) + \boldsymbol{\ell}(j)}   \right] \\
	& = x \sum\limits_{m=0}^\infty \sum\limits_{\boldsymbol{\ell} \in \bar{\N}_0^i} \left(\prod_{j\neq i}  \VEC{y}(j)^{  \boldsymbol{\ell}(j)}\right)  \boldsymbol{\mu}_i\left(\boldsymbol{\ell}+m\VEC{e}_i\right)
    \left( \widetilde{\varphi}_{i}( x, \VEC{\bar{y}}^i) \right)^m \\
	& = x\, g_i(\widetilde{\varphi}_{i}( x, \VEC{\bar{y}}^i) ,\VEC{\bar{y}}^i),
\end{split}
\end{equation*}
where in the second equality we conditioned on the offspring of the initial individual, and in the third we used Remark \ref{re: eqdist}. Since $\widetilde{\varphi}_i$ satisfies the functional equation of Proposition~\ref{Prop: pgf T_0Mvec}(i), which has a unique solution, $(\widetilde{\VEC{T}}_0, \widetilde{\VEC{M}}_1)$ has the same distribution as $(\VEC{T}_0, \VEC{M}_1)$ under $\mathbb{P}_{\VEC{e}_i}^r$.

Assume now that $\widetilde{\VEC{M}}_{0}=\VEC{a} \in \N_0^d\setminus\{ \VEC{0} \}$ and denote by $\widetilde{\varphi}_{\boldsymbol{a}}$ the generating function of $(\widetilde{\VEC{T}}_0, \widetilde{\VEC{M}}_1)$.
For $ \VEC{x}, \VEC{y} \in [0,1]^d$,
\begin{equation*}
    \begin{split}
    \widetilde{\varphi}_{\boldsymbol{a}}( \VEC{x}, \VEC{y})
    &= \mathbb{E}_{\VEC{a}}^r\left[
        \prod_{i=1}^d  \VEC{x}(i)^{ \boldsymbol{\tau}_{\VEC{a}(i)}(i) } \prod_{j=1}^d   \VEC{y}(j)^{\sum\limits_{i\neq j} \VEC{S}_{\boldsymbol{\tau}_{\VEC{a}(i)}(i)}^{(i)}(j) } \right] \\
    &=\prod_{i=1}^d \mathbb{E}^r\left[  \VEC{x}(i)^{ \boldsymbol{\tau}_{\VEC{a}(i)}(i) } \prod_{j\neq i}   \VEC{y}(j)^{\VEC{S}_{\boldsymbol{\tau}_{\VEC{a}(i)}(i)}^{(i)}(j) } \right]
        = \prod_{i=1}^d \left(\widetilde{\varphi}_{i}( \VEC{x}(i), \VEC{\bar{y}}^i)\right)^{\VEC{a}(i)},
    \end{split}
\end{equation*}
where the second equality uses Remark~\ref{rem:indrw} and the third Remark~\ref{re: eqdist}. Comparing with \eqref{eq: pgfgenstart} shows that $(\widetilde{\VEC{T}}_0, \widetilde{\VEC{M}}_1)$ and $\left(\VEC{T}_0, \VEC{M}_1\right)$ have the same distribution under $\mathbb{P}_{\VEC{a}}^r$.

Finally, by Remark \ref{re: eqdist}, for $\VEC{x},\VEC{y} \in [0,1]^d$ and $\VEC{u},\VEC{v}\in\N_0^d$,
\begin{equation*}
    \begin{split}
        \mathbb{E}_{\VEC{a}}^r &\left[
        \prod_{i=1}^d  \VEC{x}(i)^{ \widetilde{\VEC{T}}_{k}(i) } \prod_{j=1}^d   \VEC{y}(j)^{\widetilde{\VEC{M}}_{k+1}(j) } \Bigm| \widetilde{\VEC{T}}_{k-1}=\VEC{u}, \widetilde{\VEC{M}}_{k}=\VEC{v}   \right]
        =\prod_{i=1}^d \left(\widetilde{\varphi}_{i}( \VEC{x}(i), \VEC{\bar{y}}^i)\right)^{\VEC{v}(i)},
    \end{split}
\end{equation*}
so that $\left( (\widetilde{\mathbf{T}}_k,\widetilde{\mathbf{M}}_{k+1}): k \in \mathbb{N}_0   \right)$ is a Markov chain with the transition probabilities \eqref{probatranTM}.
\end{proof}

\subsection{The multitype tree of alleles} \label{sec:alleleTree}

We now introduce the $d$-type allele tree $\mathscr{A}$, which encodes the genealogy of the allelic subfamilies arising in a population evolving under mother-dependent neutral mutations. The construction follows the same spirit as the reduced tree in~\cite{Loic2016codmult}. More precisely, it is built upon the tree of alleles defined in~\cite{Bertoin10} for a population with infinitely many alleles initiated by a single individual, and adapts it to the present framework with finitely many alleles.

The adaptation requires some care, in particular, in our tree labeling. In the infinite-alleles setting the
subfamilies begotten by a given subfamily are unlabelled, and one ranks them by
decreasing size, breaking ties uniformly at random. Here each subfamily carries a
type, and the pair (size, mutant vector) does not admit a canonical total order.
We therefore rank the subfamilies by decreasing size \textit{within each type} and index them according to the Ulam-Harris tree $\U$. With this convention, the size of the subfamily identified by $u$ is given by its mark $\mathcal{B}_u$, its type is given by its mark $\mathcal{C}_u$, and the offspring vector $\VEC{d}_u$ records the numbers of mutants of each type.

For clarity and consistency with~\cite{Bertoin10}, we formulate the construction for a population started by a single individual of type $i$. By the branching property \eqref{eq: product measure} it extends to arbitrary initial populations, and when all the initial individuals share the type $i$, as in Theorem~\ref{thm:mainalles}, the $0$-th allelic generation is again a single allelic subfamily and the construction below applies verbatim with $\mathcal{B}_\varnothing = \VEC{T}_0(i)$.

A key element in our construction is $\mathscr{L}_k$, the set of mutants in the $k$-th allelic generation of $\mathcal{T}_d$, as defined in \eqref{eq:stopplinemutk} in Appendix \ref{sec:gralbrachprop}. By convention, $\mathscr{L}_0$ consists of the initial individuals. For each $k\in\N$, the individuals in $\mathscr{L}_k$ are ranked according to the breadth-first search algorithm within each allelic subfamily; see Figure~\ref{fig:treerwalle}.

We construct the multitype allele tree recursively by levels. For level zero, we set
\begin{equation} \label{eq:initAllele}
    \mathscr{A}_{\varnothing} \coloneqq (\mathcal{B}_\varnothing, \mathcal{C}_\varnothing) \coloneqq (\VEC{T}_0 (i),\, i),
\end{equation}
so that the root of the multitype allele tree records the size of the clone subfamily in the $0$-th allelic generation of $\mathcal{T}_d$ together with its color.

Now let $u\in \U$ and suppose that $\mathscr{A}_u$ and $\VEC{d}_u$ have been defined. If $\mathcal{B}_u=0$, we set $\mathcal{B}_{uv}\coloneqq 0$ for every $v$. Otherwise $\mathcal{B}_u>0$ represents the size of a clone subpopulation of the $\vert u \vert$-th allelic generation, whose founder belongs to $\mathscr{L}_{\vert u \vert}$; the $|\VEC{d}_u|$ subpopulations of the $(\vert u \vert +1)$-th allelic generation descending from it have their founders in $\mathscr{L}_{\vert u \vert +1}$. Thus we have that $\mathcal{B}_{u\ell}=0$ for all $\ell > |\VEC{d}_u|$. For each type $j \neq \mathcal{C}_u$, we list the $\VEC{d}_u(j)$ descending subfamilies in decreasing order of their sizes, ties being broken uniformly at random, and let $\mathcal{B}_{u\ell}$ be the size of the $\left(\ell-\sum_{i=1}^{j-1}\VEC{d}_u(i)\right)$-th of them.  By Remark~\ref{rem:mutdiftypmoth}, $\mathcal{C}_{u\ell}\neq \mathcal{C}_u$ for every $\ell \in \{1,\dots,|\VEC{d}_u|\}$.

We call the process $\mathscr{A} \coloneqq (\mathscr{A}_u: u \in \U_d)$, where $\mathscr{A}_u\coloneqq (\mathcal{B}_u,\mathcal{C}_u)$, the \textit{$d$-type tree of alleles}. See Figure \ref{fig:treerwalle}. 
Consequently, for all $k \in \N_0$ and $j \in [d]$, $\P^r_{\VEC{e}_i}\text{-a.s.}$
\begin{equation}\label{eq:alltreeCloMutch}
    \mathbf{T}_k (j) = \sum_{\vert u \vert  = k} \mathcal{B}_u \ind_{\{ \mathcal{C}_u = j\}}, \quad
\mathbf{M}_{k + 1} (j) = \sum_{\vert u \vert = k} \mathbf{d}_u (j)=\sum_{\vert u \vert = k}\left|(uk : \mathcal{B}_{uk}>0, \mathcal{C}_{uk}=j)\right|.
\end{equation}

As a consequence of the general branching property applied to the stopping line of mutant individuals, the $d$-type allele tree enjoys the following branching property. For $j \in [d]$ and $m\in\N_0$, we write $\bigl( (\VEC{T}_0(j),\VEC{M}_1)\bigr)^{(m\downarrow)}$ for the $m$-tuple of independent copies of $(\VEC{T}_0(j),\VEC{M}_1)$ under $\P^r_{\VEC{e}_j}$, rearranged in decreasing order of the first coordinate, ties being broken uniformly at random.

\begin{lemma}\label{lem: transalletree}
   Let $i\in [d]$ and $k\in\N_0$. Under $\P_{\VEC{e}_i}^r$, conditionally on $((\mathscr{A}_u, \mathbf{d}_u): \vert u \vert \leq k )$:
   \begin{enumerate}
       \item [(i)] the families $\bigl((\mathscr{A}_{u\ell}, \mathbf{d}_{u\ell}): \ell \in \N\bigr)$ associated with the distinct vertices $u\ell$ at level $k+1$ with $\mathcal{B}_{u\ell}>0$ are independent;
       \item [(ii)] for each such $u$, the families indexed by distinct types $j \neq \mathcal{C}_u$ are independent, and
       \[
       \bigl((\mathcal{B}_{u\ell}, \mathbf{d}_{u\ell}):  \sum\limits_{i=1}^{j-1}\VEC{d}_u(i)< \ell \leq \sum\limits_{i=1}^{j}\VEC{d}_u(i) \bigr)
       \overset{\mathcal{D}}{=}
       \bigl( (\VEC{T}_0(j),\VEC{M}_{1}) \bigr)^{(\VEC{d}_u(j) \downarrow)} .
       \]
   \end{enumerate}
\end{lemma}

\begin{proof}
The set $\mathscr{L}_{k+1}$ of mutants of the $(k+1)$-th allelic generation is a stopping line, see \eqref{eq:stopplinemutk}. 
Conditionally on  $\mathscr{F}_{\mathscr{L}_{k+1}}$, the $\sigma$-algebra defined in \eqref{eq:sigmalinea}, 
from  Proposition~\ref{prop: gralbranch} it follows that
the subtrees rooted at the individuals of $\mathscr{L}_{k+1}$ are independent, and the subtree rooted at an individual of type $j$ has the law of the genealogical tree under $\P^r_{\VEC{e}_j}$. The $\sigma$-field generated by $((\mathscr{A}_u, \mathbf{d}_u): \vert u \vert \leq k )$ is contained in $\mathscr{F}_{\mathscr{L}_{k+1}}$, and the number of individuals of $\mathscr{L}_{k+1}$ of type $j$ descending from the subfamily $u$ is $\VEC{d}_u(j)$. The claim follows, the ranking being applied within each type.
\end{proof}

\section{Asymptotic behavior of the allelic subfamilies}\label{sec:proofmain1}

The aim of this section is to prove Theorems \ref{thm:limclonemutchain} and~\ref{thm:mainalles}. Throughout, $\VEC{S}^{(n,i)}$ denotes the random walk \eqref{eq: drw} built from the offspring law $\boldsymbol{\mu}^{(n)}_i$ of \eqref{eq: mothdepproba} with mutation probability $r(n)$, and $\boldsymbol{\tau}^{(n)}_{\ell}(i)$ the corresponding hitting times \eqref{eq:tau}. We first establish the scaling limit of the underlying random walks, then that of the first step of the clone-mutant Markov chain, and finally we prove the two theorems.

\subsection{Scaling limit of the random walks}

\begin{lemma}\label{lem: rw conv}
Assume \ref{hyp:size} and \ref{hyp:mutation}. Then, as $n \to \infty$,
\begin{equation}\label{eq: RWtoBM}
\left(\left(\frac{1}{n}\VEC{S}^{(n,i)}_{\lfloor n^2t\rfloor}\right)_{t\geq 0} : i\in [d]\right)
	 \overset{\mathcal D}{\Longrightarrow }\left(\left(\left(\sigma \VEC{B}_t(i) -ct\right)\VEC{e}_i +   \frac{ct}{d-1}\bare{i}\right)_{t\geq 0}: i\in [d]\right),
\end{equation}
in the sense of weak convergence in $D(\R_+,\R^d)^{[d]}$ endowed with the product Skorokhod topology, where $(\VEC{B}(i))_{i \in [d]}$ are independent standard Brownian motions.
\end{lemma}

\begin{proof}
Fix $i \in [d]$ and write $\boldsymbol{\xi}^{(n,i)}$ for a random vector with law $\boldsymbol{\mu}^{(n)}_i$. We treat the coordinate $i$ and the coordinates $j \neq i$ separately.

\textit{The $i$-th coordinate.} Let $\VEC{S}^{(n,i)}_{k}(i)=\sum_{\ell\leq k}(\boldsymbol{\xi}^{(n,i)}_\ell(i)-1)$ with $(\boldsymbol{\xi}^{(n,i)}_\ell)_{\ell\in\N}$ i.i.d. copies of $\boldsymbol{\xi}^{(n,i)}$. Under our assumption, we have that,
\[
	\E[\boldsymbol{\xi}^{(n,i)}(i)]=1-r(n) \qquad \text{and} \qquad \Var[\boldsymbol{\xi}^{(n,i)}(i)]=r(n)(1-r(n))+(1-r(n))^2\sigma^2 \xrightarrow[n\to\infty]{} \sigma^2 .
\]
We want to use \cite[Theorem 18.2]{billingsley1999} to prove that 
\[
\sum\limits_{\ell \leq \lfloor n^2 t\rfloor}\frac{\boldsymbol{\xi}^{(n,i)}_\ell(i)-(1-r(n))}{n\sigma} \overset{\mathcal{D}}{\Longrightarrow} (\VEC{B}_t(i))_{t\geq 0},
\]
which implies that 
\[
\frac{1}{n}\VEC{S}^{(n,i)}_{\lfloor n^2 t\rfloor}(i) \overset{\mathcal{D}}{\Longrightarrow} (\sigma\VEC{B}_t(i)-ct)_{t\geq 0}.
\]
To this end, it is enough to show that
\begin{equation}\label{eq:cond1Bill}
\sum\limits_{\ell \leq \lfloor n^2 t\rfloor}\Var\left[\frac{\boldsymbol{\xi}^{(n,i)}_\ell(i)-(1-r(n))}{n\sigma}\right] \underset{n \rightarrow \infty}{\longrightarrow} t     
\end{equation}
for every $t$, and
\begin{equation}\label{eq:cond2Bill}
\sum\limits_{\ell \leq \lfloor n^2 t\rfloor}\E\left[\frac{\left(\boldsymbol{\xi}^{(n,i)}_\ell(i)-(1-r(n))\right)^2}{n^2\sigma^2} \mathbbm{1}_{\{|\boldsymbol{\xi}^{(n,i)}_\ell(i)-(1-r(n))|>\varepsilon n\sigma\} }\right] \underset{n \rightarrow \infty}{\longrightarrow} 0  
\end{equation}
for every $t$ and $\varepsilon$.
To prove condition \eqref{eq:cond1Bill}, note that 
\[
\sum\limits_{\ell \leq \lfloor n^2 t\rfloor}\Var\left[\frac{\boldsymbol{\xi}^{(n,i)}_\ell(i)-(1-r(n))}{n\sigma}\right]= 
\frac{\lfloor n^2 t\rfloor}{n^2}\frac{r(n)(1-r(n))+(1-r(n))^2\sigma^2}{\sigma^2},
\]
which goes to $t$ when $n$ tends to infinity.
On the other hand, note that
\begin{equation*}
 |\boldsymbol{\xi}^{(n,i)}_\ell(i)-(1-r(n))|\leq   \left\{
    \begin{array}{cc}
  \xi^+_\ell       & \text{ if }  \boldsymbol{\xi}^{(n,i)}_\ell(i)\geq 1 \\
   1      & \text{ if } \boldsymbol{\xi}^{(n,i)}_\ell(i)=0
    \end{array}
    \right. ,
\end{equation*}
due to $0 \leq \boldsymbol{\xi}^{(n,i)}(i)\leq \xi^{(+)}$ with $\xi^{(+)}\sim\mu^+$ not depending on $n$. Hence, condition \eqref{eq:cond2Bill} is true because
\begin{equation*}
\begin{split}
    \sum\limits_{\ell \leq \lfloor n^2 t\rfloor}\E\left[\tfrac{\left(\boldsymbol{\xi}^{(n,i)}_\ell(i)-(1-r(n))\right)^2 }{n^2\sigma^2} \right]\mathbbm{1}_{ \{|\boldsymbol{\xi}^{(n,i)}_\ell(i)-(1-r(n))|>\varepsilon n\sigma \}}  \leq 
        \frac{\lfloor n^2 t\rfloor}{n^2\sigma^2} \E\left[\left(\xi^+ +1 \right)^2 \mathbbm{1}_{ \{|\xi^+ +1|>\varepsilon \sigma n\}}\right],
\end{split}    
\end{equation*}
where the term on the right-hand side tends to zero as $n$ tends to infinity, since $\E[(\xi^{(+)})^2]<\infty$, and the claim for the $i$-th coordinate follows.

\textit{The coordinates $j\neq i$.} Here $\VEC{S}^{(n,i)}_{k}(j)$ is a sum of $k$ i.i.d. copies of $\boldsymbol{\xi}^{(n,i)}(j)$, whose mean is $r(n)/(d-1)$ and whose variance is 
\[
\Var \left[\boldsymbol{\xi}^{i}(j)\right]= \frac{r}{d-1}\left(1-\frac{r}{d-1}\right) + \frac{r^2}{(d-1)^2}\sigma^2,
\]
which is $O(r(n))=O(n^{-1})$. Consequently, for every $T>0$,
\[
	\sup_{t\leq T}\left\vert \frac{\lfloor n^2t\rfloor}{n}\frac{r(n)}{d-1}-\frac{ct}{d-1}\right\vert \xrightarrow[n\to\infty]{}0 ,
\]
while Doob's $L^2$ inequality applied to the centred walk gives
\[
	\E\left[\sup_{k\leq n^2T}\left(\VEC{S}^{(n,i)}_{k}(j)-k\frac{r(n)}{d-1}\right)^2\right]
	\leq 4 n^2 T\,\Var\left[\boldsymbol{\xi}^{(n,i)}(j)\right] = O(nT).
\]
Hence $\sup_{t\leq T}\bigl\vert n^{-1}\VEC{S}^{(n,i)}_{\lfloor n^2t\rfloor}(j)-\frac{ct}{d-1}\bigr\vert \to 0$ in $L^2$, so these coordinates converge to a deterministic limit, uniformly on compact time intervals.

Weak convergence of the $i$-th coordinate together with convergence in probability of the remaining ones to a continuous deterministic limit yields joint convergence in $D(\R_+,\R^d)$ by Slutsky's lemma \cite[Theorem 3.1]{billingsley1999}. Finally, the $d$-multidimensional random walks are independent by Remark~\ref{rem:indrw}, which yields the joint convergence over $i\in[d]$ and the independence of the limiting Brownian motions.
\end{proof}

The next lemma transfers this convergence to the hitting times. Recall from
Remark~\ref{rem:IGparam} the first-passage process $\boldsymbol{\tau}$. We write
$\boldsymbol{\tau}_s(i)\coloneqq \inf\{t\geq0 : \sigma\VEC{B}_t(i)-ct=-s\}$ for
the one built from $\VEC{B}(i)$.

\begin{lemma}\label{lem:hittingconv}
Assume \ref{hyp:size}, \ref{hyp:start} and \ref{hyp:mutation}. Then, jointly in $i \in [d]$,
\[
	\left(\boldsymbol{\theta}_n(i),\left(\VEC{U}^{(n,i)}_{\boldsymbol{\theta}_n(i)}(j), j\neq i\right)\right)
	\overset{\mathcal{D}}{\Longrightarrow}
	\left(\boldsymbol{\tau}_{\VEC{y}(i)}(i),\left( \frac{c\boldsymbol{\tau}_{\VEC{y}(i)}(i)}{d-1}, j\neq i\right)\right),
\]
where 
\[
\boldsymbol{\theta}_n(i)\coloneqq n^{-2}\boldsymbol{\tau}^{(n)}_{\VEC{a}^{(n)}(i)}(i)
\quad \text{ and } \quad \VEC{U}^{(n,i)}_t(j)\coloneqq n^{-1}\VEC{S}^{(n,i)}_{\lfloor n^2t\rfloor}(j),
\]
and the $d$ limiting variables $\boldsymbol{\tau}_{\VEC{y}(1)}(1),\dots,\boldsymbol{\tau}_{\VEC{y}(d)}(d)$ are independent, with
\[
	\boldsymbol{\tau}_{\VEC{y}(i)}(i)\sim \inverseGauss\left(\frac{\VEC{y}(i)}{c},\frac{\VEC{y}(i)^2}{\sigma^2}\right).
\]
\end{lemma}

\begin{proof}
Fix $i$ and write $W^{(n)}_t\coloneqq n^{-1}\VEC{S}^{(n,i)}_{\lfloor n^2t\rfloor}(i)$ and $W_t\coloneqq \sigma \VEC{B}_t(i)-ct$. We have then that
\begin{equation*}
    \begin{split}
\boldsymbol{\theta}_n(i)&=n^{-2}\boldsymbol{\tau}^{(n)}_{\VEC{a}^{(n)}(i)}(i)=\inf\left\{t \in n^{-2}\N_0 : W^{(n)}_t=-\frac{\VEC{a}^{(n)}(i)}{n}\right\}\\
&	=\inf\left\{t\geq0 : W^{(n)}_t\leq-\frac{\VEC{a}^{(n)}(i)}{n}\right\}.        
    \end{split}
\end{equation*}

\bigskip
By \ref{hyp:start}, $\VEC{a}^{(n)}(i)/n\to\VEC{y}(i)$. The Brownian motion with drift $W$ satisfies, almost surely and simultaneously for all levels $s>0$, $\inf\{t: W_t\leq -s\}=\inf\{t : W_t<-s\}$, and this quantity is continuous in  $s$, in particular in $s = \VEC{y}(i)$ a.s.. Therefore, the first-passage functional is a.s. continuous at $W_t$ with respect to the Skorokhod topology (see \cite[Theorem 13.6.4]{whitt2002}), and Lemma~\ref{lem: rw conv} together with the continuous mapping theorem gives
\begin{equation}\label{eq: convtau0}
	\boldsymbol{\theta}_n(i)=n^{-2}\boldsymbol{\tau}^{(n)}_{\VEC{a}^{(n)}(i)}(i)\overset{\mathcal{D}}{\Longrightarrow}\boldsymbol{\tau}_{\VEC{y}(i)}(i),
\end{equation}
jointly with the convergence \eqref{eq: RWtoBM}. If $\VEC{y}(i)=0$ the same argument gives convergence to $0$.

For $j \neq i$, we have that
\begin{equation*}
\begin{split}
\P\left( \left\vert\VEC{U}^{(n,i)}_{\boldsymbol{\theta}_n(i)}(j)\right.\right.&-\left.\left.\frac{c\boldsymbol{\tau}_{\VEC{y}(i)}(i)}{d-1}\right\vert >\varepsilon \right)\\ 
&\leq\P \left( \left\vert\ \VEC{U}^{(n,i)}_{\boldsymbol{\theta}_n(i)}(j)-\frac{c\boldsymbol{\theta}_n(i)}{d-1}\right\vert >\varepsilon/2 \right) + 
\P \left( \left\vert\frac{c}{d-1}\left( \boldsymbol{\theta}_n(i) - \boldsymbol{\tau}_{\VEC{y}(i)}(i)\right)\right\vert >\varepsilon/2 \right),        
\end{split}
\end{equation*}
where the first term  goes to zero  in probability by the uniform convergence on compact time intervals established in Lemma \ref{lem: rw conv}, and the second one goes to zero in probability by the continuous mapping theorem. The joint convergence follows from \eqref{eq: convtau0} and Slutsky's lemma.

The independence over $i\in[d]$ is that of Remark~\ref{rem:indrw}, and the law of $\boldsymbol{\tau}_{\VEC{y}(i)}(i)$ is given in Remark~\ref{rem:IGparam}.
\end{proof}
\subsection{Scaling limits for a typical clone subfamily}

\begin{proposition}\label{prop:conv-step1}
Assume \ref{hyp:size}, \ref{hyp:start} and \ref{hyp:mutation}. Let $\VEC{X}_1$ be a $d$-dimensional random vector with independent components such that
  \[
\VEC{X}_1(j) \sim \inverseGauss \left( \frac{\VEC{y}(j)}{c}, \frac{\VEC{y}(j)^2}{\sigma^2}\right), \qquad j\in [d],
    \]
and let $\VEC{Z}_1\coloneqq \VEC{A}\VEC{X}_1$, that is,
\begin{equation}\label{eq:defcomp-limt-mut}
    \VEC{Z}_1(j)=\tfrac{c}{d-1}\sum\limits_{i\neq j}\VEC{X}_1(i), \qquad j\in [d].
\end{equation}
Then:
\begin{itemize}
    \item [(i)] the following convergence in distribution holds,
\begin{equation} \label{eq:conv-step1}
\mathcal{L}\left(	\left( n^{-2}\VEC{T}_0^{(n)},n^{-1}\VEC{M}_{1}^{(n)} \right) ; \mathbb{P}_ {\VEC{a}^{(n)}}^{r(n)}  \right)
	\overset{\mathcal{D}}{\Longrightarrow}
	\left( \VEC{X}_1, \VEC{Z}_1 \right);
\end{equation}
\item [(ii)] for every $i \in [d]$, write $\Psi_i(s)=s\VEC{e}_i + \frac{c}{d-1}s\bar{\VEC{e}}^i$ for $s\in \R_+$ and $\eta^{(i)}=c^{-1}\nu\left(\Psi_i^{-1}\right)$, the image of $c^{-1}\nu$ on $\R_{+} \times\R_{+}^d$ under $\Psi_i$,
one has the weak convergence of measures on $\R_+ \times\R_+^d$,
\begin{equation}\label{eq:vagueconv}
   n\mathbb{P}^{r(n)}_{\VEC{e}_i}\left((n^{-2}T_0(i), n^{-1}\VEC{M}_1) \in \cdot \right) 
   \xlongrightarrow{w} \eta^{(i)}(\cdot).
\end{equation}
\end{itemize}
\end{proposition}

\begin{proof}
(i) By Lemma \ref{lem: simL3Ber}, for each $n\in \N$,
$
(\VEC{T}_0^{(n)},\VEC{M}_1^{(n)})\overset{\mathcal{D}}{=}
( \widetilde{\VEC{T} }_0^{(n)}, \widetilde{\VEC{M}}_{1}^{(n)}),
$
and by construction
\[
\VEC{\widetilde{T}}^{(n)}_0(i) = \boldsymbol{\tau}^{(n)}_{\VEC{a}^{(n)}(i)}(i),
\qquad
\widetilde{\VEC{M}}_{1}^{(n)}(j) = \sum_{i \neq j}\VEC{S}^{(n,i)}_{\boldsymbol{\tau}^{(n)}_{\VEC{a}^{(n)}(i)}(i)}(j), \qquad i,j \in [d].
\]
Lemma~\ref{lem:hittingconv} gives the joint convergence of these quantities, after rescaling, towards $\bigl(\boldsymbol{\tau}_{\VEC{y}(i)}(i)\bigr)_{i\in[d]}$ and $\bigl(\frac{c}{d-1}\sum_{i\neq j}\boldsymbol{\tau}_{\VEC{y}(i)}(i)\bigr)_{j\in[d]}$ respectively. Setting $\VEC{X}_1(i)\coloneqq \boldsymbol{\tau}_{\VEC{y}(i)}(i)$ yields \eqref{eq:conv-step1}, and the independence of the coordinates and their inverse Gaussian laws being those of Lemma~\ref{lem:hittingconv}.

(ii) Fix $i$ and let $(\lambda,\boldsymbol{\alpha})\in\R_+^{1+d}$. Write
\[
	\hat{f}_n((\lambda,\boldsymbol{\alpha}))\coloneqq \E^{r(n)}_{\VEC{e}_i}\left[\exp\left\{-\lambda n^{-2}\VEC{T}^{(n)}_0(i)-(\lambda,\boldsymbol{\alpha})^\top n^{-1}\VEC{M}^{(n)}_1\right\}\right] .
\]
By \eqref{eq: product measure}, under $\P^{r(n)}_{n\VEC{e}_i}$ the pair $(\VEC{T}_0^{(n)}(i),\VEC{M}_1^{(n)})$ is a sum of $n$ i.i.d. copies of the same pair under $\P^{r(n)}_{\VEC{e}_i}$; so applying part (i) with $\VEC{a}^{(n)}=n\VEC{e}_i$ and $\VEC{y}=\VEC{e}_i$, we get
\[
	\left(\hat{f}_n((\lambda,\boldsymbol{\alpha}))\right)^{n} \xrightarrow[n\to\infty]{} \E\left[e^{-\left(\lambda+\tfrac{c}{d-1}\textstyle\sum_{j\neq i}\boldsymbol{\alpha}(j)\right)\VEC{X}_1(i)}\right]= \exp\left\{-\kappa\left(\lambda+\tfrac{c}{d-1}\textstyle\sum_{j\neq i}\boldsymbol{\alpha}(j)\right)\right\} ,
\]
by \eqref{eq:IGid}. 
As we know that
\begin{equation*}
\begin{split}
\left(\hat{f}_n((\lambda,\boldsymbol{\alpha}))\right)^{n}
    =\left[\int_{\R_+\times\R_+^d}e^{-\lambda \VEC{s}(i)}e^{-\boldsymbol{\alpha}^\top\bar{\VEC{s}}^i}\mu_n(\mathrm{d}\VEC{s})\right]^{n}, 
\end{split}
\end{equation*}
where 
\[
{\mu}_n(\cdot)=\mathbb{P}^{r(n)}_{\VEC{e}_i}\left((n^{-2}T_0(i), n^{-1}\VEC{M}_1) \in \cdot \right).
\]
Taking logarithms and using the Taylor approximation,  we have that
\begin{equation}\label{eq: convtailLaplace}
\begin{split}
     n\left(1-\hat{f}_n((\lambda,\boldsymbol{\alpha}))\right)\xrightarrow[n \to \infty]{} &\kappa\left(\lambda+\tfrac{c}{d-1}\textstyle\sum_{j\neq i}\boldsymbol{\alpha}(j)\right)=
     \int_0^\infty (1-e^{- \lambda \VEC{s}}e^{-\boldsymbol{\alpha}^\top\VEC{s}})\,\eta^{(i)}(\mathrm{d}\VEC{s})
\end{split}
\end{equation}
where we use \eqref{eq:cum 1-Jirina} together with the definition of $\eta^{(i)}(\VEC{s})$. Thus, the convergence \eqref{eq: convtailLaplace} for all $(\lambda,\boldsymbol{\alpha})\in\mathbb{R}_+^{1+d}$ is equivalent to the weak convergence of the corresponding  measures (see \cite[Theorem 6.23]{kallenberg2002}).
\end{proof}

\subsection{Proof of Theorem \ref{thm:limclonemutchain}}

We first record the identification of the branching mechanism. By \eqref{eq:transitionlimit}--\eqref{eq: GMFk} with $\boldsymbol{\theta}=\VEC{0}$, for $k \geq 1$ and $\VEC{u},\boldsymbol{\lambda}\in\R_+^d$,
\begin{equation*}
    \begin{split}
\E\left[e^{-\boldsymbol{\lambda}^\top\VEC{X}_{k+1}}\mid \VEC{X}_k=\VEC{u}\right]
	&=\prod_{j=1}^d e^{-(\VEC{A}\VEC{u})(j)\kappa(\boldsymbol{\lambda}(j))}
	=\exp\left\{-\sum_{i=1}^d\VEC{u}(i)\frac{c}{d-1}\sum_{j\neq i}\kappa(\boldsymbol{\lambda}(j))\right\}\\
	&= \exp\{-\VEC{u}^\top\boldsymbol{\kappa}(\boldsymbol{\lambda})\},        
    \end{split}
\end{equation*}
with $\boldsymbol{\kappa}$ as in \eqref{eq: cum_type_j_Jir}. Thus $(\VEC{X}_k)_{k\geq1}$ is a multitype Ji\v{r}ina process with branching mechanism $\boldsymbol{\kappa}$. Moreover, $\boldsymbol{\kappa}_i$ is the Laplace exponent of the $\R_+^d$-valued subordinator obtained by placing independent inverse Gaussian subordinators with L\'evy measure $\frac{1}{d-1}\nu$ on the axes $\R_+\VEC{e}_j$, $j\neq i$.

It remains to prove the joint convergence \eqref{eq:conv-clonemut}. Since the processes are indexed by $\N_0$, it suffices to establish the convergence of the finite-dimensional distributions, and hence, by \eqref{eq:transitionlimit} and the continuity theorem for Laplace transforms, that for every $k \in \N$ and all $\boldsymbol{\lambda}_s, \boldsymbol{\theta}_s \in\R_+^d$, $s \in [k]$,
\begin{equation}\label{eq: convLaplace}
\E^{r(n)}_{\VEC{a}^{(n)}}\left[ \prod_{s=1}^k e^{- ( \boldsymbol{\lambda}_s,\boldsymbol{\theta}_s )^\top(n^{-2}\VEC{T}_{s-1}^{(n)}, \,n^{-1}\VEC{M}^{(n)}_s ) }\right]  \underset{n\rightarrow \infty}{\longrightarrow} \exp\left\{-\sum_{j=1}^d\VEC{y}(j)\,\kappa\bigl(\boldsymbol{\ell}_{k-1}(\boldsymbol{\alpha}_1)(j)\bigr)\right\},
\end{equation}
where $\boldsymbol{\ell}$ is defined in \eqref{eq:defell} from the branching mechanism $\boldsymbol{\kappa}$ and
\[
	\boldsymbol{\alpha}_s(j)\coloneqq \boldsymbol{\lambda}_s(j) + \frac{c}{d-1}\sum\limits_{i \neq j}\boldsymbol{\theta}_s(i), \qquad j \in [d], \ s\in[k].
\]
Indeed, the right-hand side of \eqref{eq: convLaplace} is the Laplace transform of the limit chain: by \eqref{eq: rec_Lap_sum_k_gen_dJir} applied to $(\VEC{X}_s)_{s\geq1}$ and the identity $\VEC{x}^\top\boldsymbol{\kappa}(\boldsymbol{\mu})=\sum_{j}(\VEC{A}\VEC{x})(j)\kappa(\boldsymbol{\mu}(j))$, which is valid for every $\VEC{x}\in\R_+^d$ and turns the initial condition $\VEC{A}\VEC{X}_0=\VEC{Z}_0=\VEC{y}$ into a nonnegative quantity.

We argue by induction on $k$. The case $k=1$ is Proposition~\ref{prop:conv-step1}(i), since $\boldsymbol{\ell}_0(\boldsymbol{\alpha}_1)=\boldsymbol{\alpha}_1$ and, by \eqref{eq:IGid},
\[
\E\left[e^{-(\boldsymbol{\lambda}_1,\boldsymbol{\theta}_1)^\top(\VEC{X}_1,\VEC{Z}_1)}\right]
=\prod_{j=1}^d\E\left[e^{-\boldsymbol{\alpha}_1(j)\VEC{X}_1(j)}\right]
=\exp\left\{-\sum_{j=1}^d\VEC{y}(j)\kappa(\boldsymbol{\alpha}_1(j))\right\} .
\]

Assume that \eqref{eq: convLaplace} holds for $k$ and for every choice of the parameters, and let us prove it for $k+1$. Let $\mathcal{F}_k$ be the $\sigma$-field generated by $\{(\VEC{T}_{s-1}^{(n)},\VEC{M}^{(n)}_{s}): s \leq k\}$. By Lemma~\ref{le:TyMMarkov},
\begin{equation}\label{eq:markovstep}
\begin{split}
&\E^{r(n)}_{\VEC{a}^{(n)}}\left[\prod_{s=1}^{k+1} e^{- ( \boldsymbol{\lambda}_{s},\boldsymbol{\theta}_{s} )^\top(n^{-2}\VEC{T}_{s-1}^{(n)}, \,n^{-1}\VEC{M}^{(n)}_{s} ) }  \right] =  \E^{r(n)}_{\VEC{a}^{(n)}}\left[ \prod_{s=1}^{k} e^{-( \boldsymbol{\lambda}_{s},\boldsymbol{\theta}_{s} )^\top(n^{-2}\VEC{T}_{s-1}^{(n)}, \,n^{-1}\VEC{M}^{(n)}_s ) }\, \Phi_n\bigl(\VEC{M}^{(n)}_{k}\bigr)\right],
\end{split}
\end{equation}
where, for $\VEC{m}\in\N_0^d$,
\[
	\Phi_n(\VEC{m})\coloneqq \E^{r(n)}_{\VEC{m}}\left[
    e^{- \boldsymbol{\lambda}_{k+1}^\top n^{-2}\VEC{T}_0^{(n)} - \boldsymbol{\theta}_{k+1}^\top n^{-1}\VEC{M}^{(n)}_1 }\right]
    = \prod_{j=1}^d \left(\E^{r(n)}_{\VEC{e}_j}\bigl[e^{-\boldsymbol{\lambda}_{k+1}^\top n^{-2}\VEC{T}_0^{(n)}-\boldsymbol{\theta}_{k+1}^\top n^{-1}\VEC{M}^{(n)}_1}\bigr]\right)^{\VEC{m}(j)},
\]
 by \eqref{eq: product measure}. Proposition~\ref{prop:conv-step1}(i) applied with $\VEC{a}^{(n)}=n\VEC{e}_j$ gives 
 \[
 \left(\E^{r(n)}_{\VEC{e}_j}\bigl[e^{-\boldsymbol{\lambda}_{k+1}^\top n^{-2}\VEC{T}_0^{(n)}-\boldsymbol{\theta}_{k+1}^\top n^{-1}\VEC{M}^{(n)}_1}\bigr]\right)^n\longrightarrow e^{-\kappa(\boldsymbol{\alpha}_{k+1}(j))},
 \]
 that is,
\[
	\varepsilon_n(j)\coloneqq -n\log \E^{r(n)}_{\VEC{e}_j}\bigl[e^{-\boldsymbol{\lambda}_{k+1}^\top n^{-2}\VEC{T}_0^{(n)}-\boldsymbol{\theta}_{k+1}^\top n^{-1}\VEC{M}^{(n)}_1}\bigr]-\kappa(\boldsymbol{\alpha}_{k+1}(j))\xrightarrow[n\to\infty]{}0 , \qquad j \in [d].
\]
Therefore
\[
	\Phi_n(\VEC{M}^{(n)}_k)=\exp\left\{-\left(n^{-1}\VEC{M}^{(n)}_k\right)^\top\bigl(\boldsymbol{\kappa}^\ast_{k+1}+\varepsilon_n\bigr)\right\},
	\qquad \boldsymbol{\kappa}^\ast_{k+1}\coloneqq \bigl(\kappa(\boldsymbol{\alpha}_{k+1}(j))\bigr)_{j\in[d]} .
\]
By the induction hypothesis the family $\left\{n^{-1}\VEC{M}^{(n)}_k\right\}_{n \in \N}$ is tight, so for every $\delta>0$ there is $K$ with $\P(\vert n^{-1}\VEC{M}^{(n)}_k\vert>K)\leq\delta$ for all large $n$. On the complementary event,
\[
	\left\vert \Phi_n(\VEC{M}^{(n)}_k)-\exp\left\{-\left(n^{-1}\VEC{M}^{(n)}_k\right)^\top\boldsymbol{\kappa}^\ast_{k+1}\right\}\right\vert
	\leq e^{K\vert\varepsilon_n\vert}-1 \xrightarrow[n\to\infty]{} 0 .
\]
Since all the factors in \eqref{eq:markovstep} are bounded by $1$, we may replace $\Phi_n(\VEC{M}^{(n)}_k)$ by \newline$\exp\{-(n^{-1}\VEC{M}^{(n)}_k)^\top\boldsymbol{\kappa}^\ast_{k+1}\}$ at the cost of an error tending to $0$. The resulting expression is of the form treated by the induction hypothesis, with $\boldsymbol{\theta}_k$ replaced by $\boldsymbol{\theta}_k+\boldsymbol{\kappa}^\ast_{k+1}$ and the other parameters unchanged. The corresponding vector $\boldsymbol{\alpha}'$ is
\[
    \boldsymbol{\alpha}'_s=\boldsymbol{\alpha}_s \quad (s<k),
    \qquad
     \boldsymbol{\alpha}'_k(j)=\boldsymbol{\alpha}_k(j)+ \frac{c}{d-1}\sum\limits_{i \neq j}\kappa(\boldsymbol{\alpha}_{k+1}(i))
     =\boldsymbol{\alpha}_k(j)+\boldsymbol{\kappa}_j(\boldsymbol{\alpha}_{k+1}) .
\]
Hence the limit is $\exp\{-\sum_j\VEC{y}(j)\kappa(\boldsymbol{\ell}_{k-1}(\boldsymbol{\alpha}'_1)(j))\}$, and it only remains to observe that, by \eqref{eq:defell},
\[
\boldsymbol{\ell}_{0}(\boldsymbol{\alpha}'_k)=\boldsymbol{\alpha}'_k=\boldsymbol{\alpha}_k + \boldsymbol{\kappa}(\boldsymbol{\alpha}_{k+1})=\boldsymbol{\ell}_{1}(\boldsymbol{\alpha}_k),
\]
so that $\boldsymbol{\ell}_{k-1}(\boldsymbol{\alpha}'_1)=\boldsymbol{\ell}_{k}(\boldsymbol{\alpha}_1)$. This is \eqref{eq: convLaplace} at rank $k+1$ and completes the proof. \hfill $\square$

\subsection{Scaling limits for the genealogical structure}

The aim of this subsection is to prove Theorem~\ref{thm:mainalles}. The key point is that the convergence  of Proposition~\ref{prop:conv-step1}(ii) can be upgraded, \textit{locally uniformly} in the number of founders, to the convergence of the ranked family of subfamilies begotten by a given subfamily. This is what allows us to condition on the previous allelic generation.

For $i \in [d]$, $m\in\N_0$ and $n \in \N$, let $\chi^{(n,i)}_1,\dots,\chi^{(n,i)}_m$ be i.i.d. copies of $(n^{-2}\VEC{T}_0(i), n^{-1}\VEC{M}_1)$ under $\P^{r(n)}_{\VEC{e}_i}$, and let
\[
	\gamma^{(n,i)}_m\coloneqq \sum_{\ell=1}^m \delta_{\chi^{(n,i)}_\ell}
\]
be the associated point measure on $\bigl(\R_+\times\R_+^d\bigr)\setminus\{\VEC{0}\}$. Recall $\Psi_i$ and $\eta^{(i)}(\cdot)$ from Proposition~\ref{prop:conv-step1}(ii).

\begin{lemma}\label{lem:atoms}
Assume \ref{hyp:size} and \ref{hyp:mutation}, and fix $j \in [d]$ and $B>0$. Then, for every continuous compactly supported $g: (\R_+\times\R_+^d)\to\R_+$,
\begin{equation}\label{eq:unifLaplacefunctional}
	\sup_{0\leq m \leq Bn}\left\vert \E\left[e^{-\langle \gamma^{(n,i)}_m,\, g\rangle}\right]-\exp\left\{-\frac{m}{n}\int_{\R_+\times\R_+^d}\left(1-e^{-g(\VEC{s})}\right)\eta^{(i)}(\mathrm{d}\VEC{s})\right\}\right\vert \xrightarrow[n\to\infty]{}0 .
\end{equation}
Consequently, if $(m_n)$ are $\N_0$-valued random variables, independent of the arrays $(\chi^{(n,j)}_\ell)$, such that $m_n/n \Rightarrow \beta$ for some $\R_+$-valued random variable $\beta$ with $(m_n/n)$ tight, then $\gamma^{(n,j)}_{m_n}$ converges in distribution, for the weak topology, to a Cox process directed by $\beta\,\eta^{(j)}$. In particular, the ranked family converges:
\begin{equation}\label{eq:conv step1 atom}
    \left(\bigl(n^{-2}\VEC{T}^{(n)}_0(j), n^{-1}\VEC{M}^{(n)}_1\bigr)\right)^{(m_n\downarrow )} \overset{\mathcal{D}}{\Longrightarrow} \left(\Psi_j(\beta_1), \Psi_j(\beta_2),\dots\right),
\end{equation}
where $\beta_1\geq\beta_2\geq\cdots$ are the atoms of a Poisson random measure on $(0,\infty)$ with intensity $\beta\, c^{-1}\nu$, ranked in decreasing order and completed by zeros.
\end{lemma}

\begin{proof}
Write $\varepsilon^i_n\coloneqq 1-\E[e^{-g(\chi^{(n,i)}_1)}]$. Since $g$ is bounded, continuous and compactly supported away from the origin, Proposition~\ref{prop:conv-step1}(ii) gives $n\varepsilon^i_n \to \int_{\R_+\times\R_+^d}(1-e^{-g(\VEC{s})})\eta^{(i)}(\mathrm{d}\VEC{s})$; in particular $\varepsilon^i_n=O(n^{-1})$. By independence, $\E[e^{-\langle\gamma^{(n,i)}_m, g\rangle}]=(1-\varepsilon^i_n)^m$, and
\[
	\sup_{m\leq Bn}\left\vert m\log(1-\varepsilon^i_n)+\frac{m}{n}\,n\varepsilon^i_n\right\vert
	\leq B n \left\vert \log(1-\varepsilon^i_n)+\varepsilon^i_n\right\vert = O(n^{-1}) ,
\]
hence goes to zero uniformly with $n$.
From here we deduce \eqref{eq:unifLaplacefunctional}, using also that 
\begin{equation*}
    \sup_{m \leq Bn}\left\vert e^{-\frac{m}{n}n\varepsilon_n}-\exp\left\{-\frac{m}{n}\int_{\R_+\times\R_+^d}(1-e^{-g(\VEC{s})})\eta^{(i)}(\mathrm{d}\VEC{s})\right\}\right\vert \to 0 .
\end{equation*}

For the second assertion, we condition on $m_n$: by \eqref{eq:unifLaplacefunctional} and the hypothesis of tightness of $(m_n/n)$,
\[
    \E[e^{-\langle\gamma^{(n,j)}_{m_n},g\rangle}]-\E\left[\exp\left\{-\frac{m_n}{n}\int_{\R_+\times\R_+^d}(1-e^{-g(\VEC{s})})\eta^{(i)}(\mathrm{d}\VEC{s})\right\}\right] \to 0,
\]
and the right-hand side converges to $\E\left[\exp\left\{-\beta\int_{\R_+\times\R_+^d}(1-e^{-g(\VEC{s})})\eta^{(i)}(\mathrm{d}\VEC{s})\right\}\right]$ by the continuous mapping theorem. This identifies the limit as the announced Cox process \cite[Chapter 4]{kallenbergRM}. Finally, $\eta^{(j)}$ is carried by the half-line $\{\Psi_j(s):s>0\}$ and $\Psi_j$ is increasing, so the atoms of the limit are of the form $\Psi_j(\beta_\ell)$ with $(\beta_\ell)$ the ranked atoms of a Poisson random measure with intensity $\beta c^{-1}\nu$. The passage from vague convergence to the convergence of the ranked sequences is standard, see \cite[Section 3.3]{resnick2007}.
\end{proof}

\begin{remark}
Lemma~\ref{lem:atoms} is the analogue of \cite[Lemma 4.6]{BlancasRivero}. Note the
normalization: the number of founders must be of order $n$, not of order
$nr(n)$, and the limiting intensity is $\beta c^{-1}\nu$, so that the constant
$c$ from \ref{hyp:mutation} appears twice, once through the number of founders
and once through $\nu$.
\end{remark}
\subsubsection*{Proof of Theorem \ref{thm:mainalles}}

By the monotone class theorem it suffices to prove that, for every $k \in \N$ and all bounded continuous functions $f_1,\dots,f_k$,
\begin{equation}\label{convfindimgen}
	\mathbb{E}_{\VEC{a}^{(n)}}^{r(n)}  \left[  \prod_{\ell=1}^k  f_\ell\Bigl(
	\bigl(  n^{-2}\mathcal{B}^{(n)}_u,\, \mathcal{C}^{(n)}_u,\, n^{-1}\VEC{d}^{(n)}_u   \bigr)
	: |u| \leq \ell \Bigr)
	\right]
	\xrightarrow[n\rightarrow \infty] {}
	\mathbb{E}  \left[  \prod_{\ell=1}^k  f_\ell\bigl(
	(\mathcal{Y}_u, \mathcal{C}_u, \boldsymbol{\mathcal{Z}}_u) : |u| \leq \ell \bigr) \right],
\end{equation}
where on the right-hand side $(\mathscr{Y}_u : u \in \U_d)$ is the tree-indexed multitype Ji\v{r}ina process of Definition~\ref{def:multi-treeCSBP} with root type $i$ and root law $\inverseGauss(y/c, y^2/\sigma^2)$.

We proceed by induction on $k$. For $k=1$, only the root and its offspring vector are involved: $\mathcal{B}^{(n)}_\varnothing=\VEC{T}^{(n)}_0(i)$, $\mathcal{C}^{(n)}_\varnothing=i$ and $\VEC{d}^{(n)}_\varnothing=\VEC{M}^{(n)}_1$, so Proposition~\ref{prop:conv-step1}(i) applied with $\VEC{y}=y\VEC{e}_i$ gives
\[
	\left(n^{-2}\mathcal{B}^{(n)}_\varnothing, \mathcal{C}^{(n)}_\varnothing, n^{-1}\VEC{d}^{(n)}_\varnothing\right)
	\overset{\mathcal{D}}{\Longrightarrow}
	\left(\mathcal{Y}_\varnothing, i, \tfrac{c}{d-1}\mathcal{Y}_\varnothing \bare{i}\right),
	\qquad \mathcal{Y}_\varnothing \sim \inverseGauss\left(\frac{y}{c},\frac{y^2}{\sigma^2}\right),
\]
as required.

Assume the convergence \eqref{convfindimgen} holds at rank $k$, and let
$\mathcal{G}^{(n)}_k\coloneqq \sigma\bigl( (\mathscr{A}^{(n)}_u, \VEC{d}^{(n)}_u) : | u | \leq  k \bigr)$.
Conditioning on $\mathcal{G}^{(n)}_k$ and applying Lemma \ref{lem: transalletree}, the left-hand side of \eqref{convfindimgen} at rank $k+1$ equals
\[
\mathbb{E}_{\VEC{a}^{(n)}}^{r(n)}   \left[ \prod_{\ell=1}^k f_\ell\Bigl(
\bigl(  n^{-2}\mathcal{B}^{(n)}_u, \mathcal{C}^{(n)}_u, n^{-1}\VEC{d}^{(n)}_u   \bigr)
  : |u| \leq \ell \Bigr) \;
 F_n\Bigl(\bigl(\VEC{d}^{(n)}_u,\mathcal{C}^{(n)}_u\bigr):|u|=k\Bigr)   \right],
\]
where 
\(
	F_n \coloneqq \E\left[f_{k+1}\Bigl(\bigl(\ldots\bigr) : |u|\leq k+1\Bigr)\Bigm| \mathcal{G}^{(n)}_k\right]
\)
is obtained by inserting, independently over the vertices $u$ at level $k$ with $\mathcal{B}_u^{(n)}>0$ and over the types $j \neq \mathcal{C}^{(n)}_u$, the ranked family of $\VEC{d}^{(n)}_u(j)$ independent copies of $(n^{-2}\VEC{T}_0(j), n^{-1}\VEC{M}_1)$ under $\P^{r(n)}_{\VEC{e}_j}$.
By the induction hypothesis, $n^{-1}\VEC{d}^{(n)}_u(j)\overset{\mathcal{D}}{\Longrightarrow} \boldsymbol{\mathcal{Z}}_u(j)= \frac{c}{d-1}\mathcal{Y}_u$ for every $j \neq \mathcal{C}_u$ and $|u|=k$, jointly with the rest of the information given by the first $k$ levels; in particular the family $(n^{-1}\VEC{d}^{(n)}_u(j))_n$ is tight. Lemma~\ref{lem:atoms}, applied with $m_n=\VEC{d}^{(n)}_u(j)$ and $\frac{c}{d-1}\mathcal{Y}_u=\beta$, 
shows that the inserted families converge jointly to independent ranked sequences $\bigl(\Psi_j(\beta_{\ell})\bigr)_{\ell\in\N}$, 
where $(\beta_{\ell} : \ell\in\N)$  are the ranked atoms of a Poisson random measure with intensity $
	\frac{c}{d-1}\mathcal{Y}_u\,c^{-1}\nu = \frac{1}{d-1}\mathcal{Y}_u\,\nu .$
This is exactly the offspring law prescribed by Definition~\ref{def:multi-treeCSBP}(iii); moreover, $\Psi_j(\beta_{\ell})$ records the mass $\mathcal{Y}_{u\ell}=\beta_\ell$ together with the mutant vector $\boldsymbol{\mathcal{Z}}_{u\ell}=\frac{c}{d-1}\beta_{\ell}\bare{j}$, and the type $j$ is recorded by the label. Combining this with the induction hypothesis and the boundedness of $f_1,\dots,f_{k+1}$ yields \eqref{convfindimgen} at rank $k+1$, and the theorem follows. \hfill $\square$

\appendix

\section{The general branching property for multitype BGW processes}\label{sec:gralbrachprop}

In this section, we follow \cite{jagers2008} and introduce concepts originally discussed in \cite{Chauvin}.

We consider a $d$-type BGW process with neutral mutations and genealogy represented by $\mathcal{T}_d$, a marked planar tree with colored vertices and marked edges, which was presented at the beginning of Section~\ref{sec: representations}.

If $u\in \mathcal{T}_d$ is an individual of type $j$, then we write $\sigma_u=j$. If the individual $u$ has $m\in \N_0$ mutations in its ancestral line, we write $\mathfrak{m}_u=m$.

We use the Ulam--Harris notation $\U$ to enumerate the individuals in the tree. For example, consider $u=(u_1,\dots ,u_{n-1},u_n)\in \U$. This individual is the $u_n$-th child of the $u_{n-1}$-th child, and so forth, up to the $u_1$-th child of the root, denoted by $\varnothing$. This notation allows us to determine that the generation in which $u$ resides is $|u|=n$. To define ancestry relationships within the tree, let $u,v\in \U$. We say that $v$ is an ancestor of $u$ (or $u$ is a descendant of $v$) if there exists $w\in \U$ such that $u=vw$. This relationship is denoted by $v\preceq u $ (or $u\succeq v$). Observe that $u\preceq u$, for any $u\in \U$.

For every $u\in \U,$ let $(\Omega_u, \mathscr{G}_u)$ be a \textit{life space}, so that $\omega \in \Omega_u$ is a possible life career of an individual. Any individual property, such as its type, the number of children of each type and the marks along the edges connecting to its children, can be viewed as a measurable function with respect to $\mathscr{G}_u$. Indeed, for every $k\in \N$ the function $\sigma(k):\Omega_u \rightarrow [d]$ is the type of the $k$-th child of individual $u$, and $\mathfrak{m}(k):\Omega_u \rightarrow \{0,1\}$ indicates whether there is a mark along the edge connecting the individual $u$ to its $k$-th child. If $u \in \mathcal{T}_d \subset \U$ has colour $\sigma_u$ and mark $\mathfrak{m}_u$, then, for all $\omega \in \Omega_u$, $\sigma_{uk}(\omega)=\sigma(k)(\omega)$ and $\mathfrak{m}_{uk}(\omega)= \mathfrak{m}_{u} + \mathfrak{m}(k)(\omega)$, provided that the individual $uk$ was actually born, i.e. that its birth time $\mathfrak{t}_{uk}(\omega)$ is finite.
We denote by $\mathscr{S}$ the countably generated $\sigma$-algebra associated with the type-mark space $S\coloneqq[d] \times \N_0$.
The population space is then defined by $(\Omega,\mathscr{G})\coloneqq (S \times \Omega_u^{\U}, \mathscr{S}\otimes \mathscr{G}_u^{\U})$. For $u\in \U$, $U_u$ is the projection $\Omega \rightarrow \Omega_u$.

Let $L\subset \U$. If for all $u,v\in L$ with $u\neq v$ we have $v\nprec u$, we say that $L$ is a \textit{line}. Hence, if we consider two different vertices in a line $L$, then neither can be a descendant of the other.
Roughly speaking, a line is a perpendicular cut of the genealogical lines, taking at most one individual from each of them.
Associated with a line $L\subset \U$ is its $\sigma$-algebra
\begin{equation}\label{eq:sigmalinea}
    \mathscr{F}_L \coloneqq \mathscr{S} \otimes \sigma( U_u: u\notin \mathbf{Pr} L),
\end{equation}
where $\mathbf{Pr} L\coloneqq\{v\in \U: \exists u\in L \mbox{ with } u\preceq v\}$. Hence, $\mathscr{F}_L$ carries the information on all individuals who do not have ancestors in the line. Note that if $L \preceq M$, then $\mathbf{Pr} M \subseteq \mathbf{Pr} L$, so $\mathscr{F}_L \subseteq \mathscr{F}_M$, and therefore $(\mathscr{F}_L: L \subset \U)$ is a filtration under the relation $\preceq$.

Formally, a \textit{stopping line} $\mathcal{L}$ is a random line such that the event $\{\mathcal{L} \preceq L  \}\in \mathscr{F}_L$ for every $L\subset \U$. A simple example is $\mathcal{L}=\{ u \}$ with $u\in\U$. The stopping line $\mathcal{L}^{(n)}\coloneqq\{u\in \U: |u|=n\}$, $n\in \N_0$, is commonly employed to establish the branching property. In this work we use the stopping line consisting of the individuals with $m$ marks in their ancestral line and type different from that of their mother, that is,
\begin{equation}\label{eq:stopplinemutk}
\mathscr{L}_m\coloneqq \{ u\in \U: \mathfrak{m}_u=m \text{ and } \sigma_v\neq \sigma_u,\, v \text{ the mother of } u\}.
\end{equation}

For every $i\in[d]$ and $r\in(0,1)$, the genealogical random tree $\mathcal{T}_d$ can be defined on $(\Omega,\mathscr{G})$, namely $\mathcal{T}_d = \mathcal{R}$, where $\mathcal{R}\coloneqq \{u \in \U:\mathfrak{t}_u <\infty\}$ is the set of \textit{realized} individuals, those actually born. In this case, the mark $\mathfrak{m}_\varnothing$ of the root is always zero.

Consider the mapping
\[
S_u=(\sigma_u,\mathfrak{m}_u, U_{\mathbf{Pr}\, u}): S\times \U \rightarrow S\times \U,
\]
 which makes $u$ the new root, and let $T_u = S^{-1}_u$ be its inverse.

   Let $\mathcal{T}^{(u)}_d$ denote the subtree rooted at $u$.
   By the branching property, it is well known that, conditionally on the number of individuals in a given generation, the subtrees rooted at those individuals are independent. The following is a stronger statement.
\begin{proposition}[\cite{jagers2008}, Equation 1] \label{prop: gralbranch}
Let $\mathcal{L}$ be a stopping line for a $d$-type BGW tree $\mathcal{T}_d$ with mother-dependent neutral mutations. Conditionally on $\mathscr{F}_\mathcal{L}$, the subtrees $\{ \mathcal{T}^{(x)}_d: x\in \mathcal{L} \}$ are independent, and for each $x\in \mathcal{L}$ the tree $\mathcal{T}^{(x)}_d$ has the same law as the genealogical tree associated with $\{  (\VEC{Y}_k, k \in \mathbb{N}_0); \mathbb{P}_{\VEC{e}_{\sigma_x}}^{r}\}$. In other words, for any family of events $\{A_x : x \in \mathcal{L}\} \subseteq \mathscr{G}$,
\[
\mathbb{P}_{\VEC{e}_i}^r
\left( \bigcap_{x\in \mathcal{L}}  T_xA_x \Bigm| \mathscr{F}_\mathcal{L}\right) =  \prod_{x\in \mathcal{L}}\mathbb{P}_{ \VEC{e}_{\sigma_x} }^r\left(  A_x \right).
\]
\end{proposition}

\section{Proofs of the results of Section \ref{sec:bases}}\label{app:proofsim}

\subsection{Proof of Proposition \ref{Prop: pgf T_0Mvec} }\label{proof: pgf T_0Mvec}

\begin{enumerate}

\item [(i)]
Conditioning on the random vector $\boldsymbol{\xi}^i$ of the offspring of the initial individual of type $i$, equation~\eqref{eq:MGFone} reads
\begin{equation}\label{eq:torre1}
\varphi_{i}( x, \VEC{\bar{y}}^i)
=\sum\limits_{\VEC{v} \in \mathbb{N}_0^d} \boldsymbol{\mu}_i(\VEC{v})\, \mathbb{E}_{\VEC{e}_i}^r\left[ x^{\VEC{T}_0(i)} \prod_{j\neq i}   \VEC{y}(j)^{ \VEC{M}_1(j) }  \Bigm| \boldsymbol{\xi}^{i} = \VEC{v} \right],
\qquad x\in [0,1], \, \VEC{y}\in [0,1]^d.
\end{equation}
It remains to compute the conditional expectation. Given that the initial individual has $\VEC{v}(i)$ children of type $i$, let $(\VEC{T}_{0,1}(i),\VEC{M}_{1,1}) ,\dots,(\VEC{T}_{0,\VEC{v}(i)}(i),\VEC{M}_{1,\VEC{v}(i)})$ be $\VEC{v}(i)$ i.i.d. copies of the random vector $(\VEC{T}_0(i),\VEC{M}_1)$, obtained by the branching property. Then
\begin{align*}
\mathbb{E}_{\VEC{e}_i}^r\left[  x^{\VEC{T}_0(i)} \prod_{j\neq i}   \VEC{y}(j)^{ \VEC{M}_1(j) }
\Bigm| \boldsymbol{\xi}^i = \VEC{v}
\right]
&  = \mathbb{E}_{\VEC{e}_i}^r\left[  {x}^{1 +\sum\limits_{k=1}^{\VEC{v}(i)}\VEC{T}_{0,k}(i)} \prod_{j\neq i}    \VEC{y}(j)^{\VEC{v}(j)+\sum\limits_{k=1}^{\VEC{v}(i)} \VEC{M}_{1,k}(j)}  \right] \\
&  = {x} \left(\varphi_i( {x}, \VEC{\bar{y}}^i )\right)^{\VEC{v}(i)} \prod_{j\neq i} \VEC{y}(j)^{ \VEC{v}(j)}, \qquad \VEC{v} \in \mathbb{N}_0^d.
\end{align*}
 Substituting into~\eqref{eq:torre1} and comparing with~\eqref{eq: g_i} gives~\eqref{eq: pgfTM_i}.

For each $i\in[d]$ and each fixed $\VEC{\bar{y}}^i \in (0,1]^d$, define $g_i^*(z)\coloneqq g_i(z, \VEC{\bar{y}}^i)$.
This function is analytic in a neighborhood of zero, being a probability generating function, and $g_i^*(0) \neq 0$. Indeed, suppose that $g_{i}(0, \VEC{\bar{y}}^i)=0$ for some $\VEC{\bar{y}}^i \in (0,1]^d$. Writing
\[
\beta_k \coloneqq \sum_{\VEC{v} \in \bar{\N}^i_0 } \boldsymbol{\mu}_i\left(\VEC{v}+k\VEC{e}_i\right), \qquad k \in \N_0,
\qquad \text{so that} \qquad
\E_{\VEC{e}_i}^r[ \boldsymbol{\xi}^{i}(i) ] =\sum_{k\in\N} k \beta_k,
\]
the assumption forces $\beta_0=0$, and thus
$
    \E_{\VEC{e}_i}^r[ \boldsymbol{\xi}^{i}(i) ]\geq \beta_1+2(1-\beta_1)=2-\beta_1 .
$
Since \ref{hyp:size} implies $ \mathbb{E}_{\VEC{e}_i}^r [ \boldsymbol{\xi}^{i}(i) ]=1-r$, we get $\beta_1 \geq 1+r> 1$, which is impossible.

By the Lagrange inversion formula (see for instance~\cite[Theorem 6.2]{pitman2006combinatorial}), for each $i\in[d]$ and each fixed $\VEC{\bar{y}}^i \in (0,1]^d$, the equation
$
\varphi^*_i(x)= xg^*_i(\varphi^*_i(x))
$
has a unique analytic solution in a neighborhood of $0$, whose expansion in powers of $x$ satisfies
\begin{equation}\label{eq: LIF}
[x^m](\varphi_i(x,\VEC{\bar{y}}^i))^{q} = \frac{q}{m}[x^{m-q}](g_i(x,\VEC{\bar{y}}^i))^m,
\qquad m \geq q \geq 1;
\end{equation}
that is, the coefficient of $x^m$ in $(\varphi_i(x,\VEC{\bar{y}}^i))^{q}$ equals $\frac{q}{m}$ times the coefficient of $x^{m-q}$ in $(g_i(x,\VEC{\bar{y}}^i))^m$.
Indeed, for $x \in (0,1)$ and $\VEC{\bar{y}}^i \in (0,1]^d$ we may write
\[
	{g}_i(x, \VEC{\bar{y}}^i)=\sum\limits_{k=0}^{\infty}a_{i,k}(\VEC{\bar{y}}^i)x^{k}, \qquad \text{where } a_{i,k}(\VEC{\bar{y}}^i)= \sum\limits_{\boldsymbol{\ell} \in \bar{\N}_0^i} \boldsymbol{\mu}_i(\boldsymbol{\ell} + k\VEC{e}_i)
    \prod\limits_{j\neq i} \VEC{y}(j)^{\boldsymbol{\ell}(j)},
\]
and, since $x \mapsto g_i(x,\VEC{\bar{y}}^i)$ is the generating function of the finite measure $\alpha_i = (a_{i,k}(\VEC{\bar{y}}^i) : k \in \mathbb{N}_0)$, the generating function of $\alpha^{*m}_i$ is
\[
	x \mapsto (g_i(x,  \VEC{\bar{y}}^i))^m=
 	\sum\limits_{k=0}^{\infty}x^k \sum\limits_{\boldsymbol{\ell} \in \bar{\N}_0^i }\boldsymbol{\mu}_i^{*m}(\boldsymbol{\ell} + k\VEC{e}_i)\prod\limits_{j\neq i} \VEC{y}(j)^{\boldsymbol{\ell}(j)} .
\]
Hence, by \eqref{eq: LIF} with $q=\VEC{a}(i)$, for fixed $\VEC{\bar{y}}^i \in (0,1]^d$ and $x \in (0,1)$,
\begin{equation}\label{eq:pgf TiMexpl}
\varphi_{\VEC{a}(i)\VEC{e}_i}(x,\VEC{\bar y}^i)
=(\varphi_i(x,\VEC{\bar y}^i))^{\VEC{a}(i)}
= \sum\limits_{m=\VEC{a}(i)}^{\infty} \frac{\VEC{a}(i)}{m}\sum\limits_{\boldsymbol{\ell} \in\bar{\N}_0^i}\boldsymbol{\mu}_i^{*m}(\boldsymbol{\ell} + (m-\VEC{a}(i))\VEC{e}_i)\prod\limits_{j\neq i} \VEC{y}(j)^{\boldsymbol{\ell}(j)}x^m,
\end{equation}
while by definition
\[
\varphi_{\VEC{a}(i)\VEC{e}_i}(x,\VEC{\bar y}^i)
=\sum\limits_{m=\VEC{a}(i)}^{\infty}\sum\limits_{\boldsymbol{\ell} \in \bar{\N}_0^i}\mathbb{P}_{\VEC{a}(i)\VEC{e}_i}^r(\VEC{T}_0(i)=m,\VEC{M}_1=\boldsymbol{\ell})\,x^{m}\prod\limits_{j\neq i}\VEC{y}(j)^{\boldsymbol{\ell}(j)}.
\]
Identifying coefficients yields
\[
	\mathbb{P}_{\VEC{a}(i)\VEC{e}_i}^r(\VEC{T}_0(i)=m,\VEC{M}_1= \boldsymbol{\ell}) = \frac{\VEC{a}(i)}{m}\boldsymbol{\mu}_i^{*m}(\boldsymbol{\ell} + (m-\VEC{a}(i))\VEC{e}_i), \qquad \boldsymbol{\ell} \in \bar{\N}_0^i .
\]

 \item [(ii)] We now compute the generating function of $(\VEC{T}_0, \VEC{M}_1)$ under $\mathbb{P}_{\VEC{a}}^{r}$.
By the branching property, if $\VEC{Y}_0=\VEC{a}\in\N_0^d\setminus\{\VEC{0}\}$ then, for every $(\VEC{x},\VEC{y}) \in [0,1]^d \times [0,1]^d$,
\begin{equation} \label{eq: varphi_a}
\begin{split}
    \varphi_{\VEC{a}}( \VEC{x}, \VEC{y})
     &= \mathbb{E}_{\VEC{a}}^r
        \left[ \prod_{j=1}^d \VEC{x}(j)^{ \VEC{T}_0(j) }  \VEC{y}(j)^{ \VEC{M}_1(j) } \right]
	= \prod_{i=1}^{d} \prod_{q=1}^{\VEC{a}(i)}  \mathbb{E}_{\boldsymbol{e}_i}^r\left[ \VEC{x}(i)^{ \VEC{T}_{0,q}(i) } \prod_{j\neq i}  \VEC{y}(j)^{ \VEC{M}_{1,q}(j) } \right]  \\
	&=\prod_{i=1}^{d}   \left(\mathbb{E}_{\boldsymbol{e}_i}^r\left[ \VEC{x}(i)^{ \VEC{T}_{0}(i) } \prod_{j\neq i}   \VEC{y}(j)^{ \VEC{M}_{1}(j) } \right]\right)^{\VEC{a}(i)}
	=\prod_{i=1}^d \left(\varphi_{i}( \VEC{x}(i), {\VEC{\bar y}}^i) \right) ^{\VEC{a}(i)},
\end{split}
 \end{equation}
where, for each $i \in [d]$, the families $\{(\VEC{T}_{0,q}(i),\VEC{M}_{1,q})\}_{q \leq \VEC{a}(i) }$ are i.i.d. copies of $(\VEC{T}_0(i),\VEC{M}_{1})$ under $\P^r_{\VEC{e}_i}$. This is \eqref{eq: pgfgenstart}.

Moreover, from~\eqref{eq:pgf TiMexpl} and \eqref{eq: varphi_a}, for $\VEC{a} \in \mathbb{N}_0^d\setminus\{\VEC{0}\}$ and $\VEC{x}, \VEC{y}\in [0,1]^d$,
\begin{align*}
    &\varphi_{\VEC{a}}(\VEC{x}, \VEC{y})
    = \prod_{i=1}^d \left[\sum\limits_{\VEC{m}(i)=\VEC{a}(i)}^{\infty} \frac{\VEC{a}(i)}{\VEC{m}(i)}\sum\limits_{\VEC{w}_i \in\bar{\N}_0^i}\boldsymbol{\mu}_i^{*\VEC{m}(i)}(\VEC{w}_i + (\VEC{m}(i)-\VEC{a}(i))\VEC{e}_i)\prod\limits_{j\neq i} \VEC{y}(j)^{\VEC{w}_i(j)}\, \VEC{x}(i)^{\VEC{m}(i)}\right]\\
	&\quad= \sum\limits_{\VEC{m}\boldsymbol{\geq} \VEC{a}} \ \sum\limits_{\boldsymbol{\ell} \in \mathbb{N}_0^d}
	\left[\prod\limits_{i=1}^d \VEC{x}(i)^{\VEC{m}(i)}\VEC{y}(i)^{\boldsymbol{\ell}(i)}\right]
	\sum\limits_{(\VEC{w}_1,\dots,\VEC{w}_d) \in W_{\boldsymbol{\ell}}} \ \prod_{i=1}^d \frac{\VEC{a}(i)}{\VEC{m}(i)} \boldsymbol{\mu}_i^{*\VEC{m}(i)}(\VEC{w}_i + (\VEC{m}(i)-\VEC{a}(i))\VEC{e}_i) ,
\end{align*}
where $ W_{\boldsymbol{\ell}} = \{ (\VEC{w}_1,\dots,\VEC{w}_d) \in \bar{\N}_0^1 \times \dots \times \bar{\N}_0^d : \sum_{i=1}^d \VEC{w}_i = \boldsymbol{\ell} \}$; we used that the coefficient of $\prod_i\VEC{y}(i)^{\boldsymbol{\ell}(i)}$ in a product of $d$ such series is obtained by summing over all decompositions of $\boldsymbol{\ell}$.
On the other hand, by definition,
\[
	\varphi_{\VEC{a}}(\VEC{x}, \VEC{y})= \sum\limits_{\VEC{m}\boldsymbol{\geq}\VEC{a}}\sum\limits_{\boldsymbol{\ell} \in \mathbb{N}_0^d} \mathbb{P}_{\VEC{a}}^r\left(\VEC{T}_0=\VEC{m}, \VEC{M}_1 = \boldsymbol{\ell}\right)\prod\limits_{i=1}^d \VEC{x}(i)^{\VEC{m}(i)}\VEC{y}(i)^{\boldsymbol{\ell}(i) },
\]
and comparing the two expressions gives the announced formula.
\end{enumerate}
\hfill $\square$

\subsection{Proof of Corollary~\ref{cor:ext}} \label{proof:ext}

Throughout, $\VEC{a} \in \mathbb{N}^d_0\setminus\{ \VEC{0}\}$.

\begin{enumerate}
    \item [(i)]
First,
\[
\mathbb{E}_{\VEC{a}}^r\left[\VEC{T}_0(i)\right]
    =\frac{\partial\varphi_\VEC{a}}{\partial \VEC{x}(i)}(\VEC{x},\VEC{y})
    \Big|_{\VEC{x}=\VEC{1}, \VEC{y}=\VEC{1}} ,
\]
and by \eqref{eq: varphi_a},
\[
	\frac{\partial\varphi_\VEC{a}}{\partial \VEC{x}(i)}(\VEC{x},\VEC{y})
    =\left[\prod\limits_{j \neq i}\left(\varphi_j
     (\VEC{x}(j),{\VEC{\bar y}}^j) \right)^{\VEC{a}(j)} \right]
    \VEC{a}(i) (\varphi_i(\VEC{x}(i), {\VEC{\bar y}}^i))^{\VEC{a}(i) -1}
    \frac{\partial\varphi_i}{\partial \VEC{x}(i)}(\VEC{x}(i),{\VEC{\bar y}}^i).
\]
Differentiating \eqref{eq: pgfTM_i} with respect to $\VEC{x}(i)$ we get
\[
\frac{\partial\varphi_i}{\partial \VEC{x}(i)}(\VEC{x}(i),{\VEC{\bar y}}^i)
= g_i\left(\varphi_i(\VEC{x}(i),{\VEC{\bar y}}^i),{\VEC{\bar y}}^i\right)
\left[ 1 - \VEC{x}(i)\frac{\partial g_i}{\partial \VEC{s}(i)}\left(\varphi_i(\VEC{x}(i),{\VEC{\bar y}}^i),{\VEC{\bar y}}^i\right)  \right]^{-1} .
\]
Since
\begin{equation}\label{eq: expiipgf}
\mathbb{E}_{\VEC{e}_i}^r\left[ \boldsymbol{\xi}^{i}(i) \right]
    =\frac{\partial g_i}{\partial \VEC{s}(i)}( \VEC{s}(i),\VEC{\bar s}^i)\Big|_{\VEC{s}=\VEC{1}}=1-r, \qquad i\in[d],
\end{equation}
and $\varphi_i(1,\VEC{1})=g_i(\VEC{1})=1$, we conclude that
\[
   \mathbb{E}_{\VEC{a}}^r\left[\VEC{T}_0(i)\right]
 =  \VEC{a}(i)\left[ 1 -  \mathbb{E}_{\VEC{e}_i}^r\left[ \boldsymbol{\xi}^{i}(i) \right]\right]^{-1}=\frac{\VEC{a}(i)}{r}.
\]
On the other hand,
$
\mathbb{E}_{\VEC{a}}^r[\VEC{M}_1(i)]
    =\frac{\partial\varphi_\VEC{a}}{\partial \VEC{y}(i)}(\VEC{x},\VEC{y})
    \big|_{\VEC{x}=\VEC{1}, \VEC{y}=\VEC{1}}$,
and by \eqref{eq: varphi_a},
\begin{equation*}
\begin{split}
	&\frac{\partial\varphi_\VEC{a}}{\partial \VEC{y}(i)}(\VEC{x},\VEC{y})\\
& \quad = ( \varphi_i(\VEC{x}(i),\VEC{\bar y}^i) )^{\VEC{a}(i)}
   \sum_{j\neq i}
    \left[
    \prod\limits_{k \in[d] \setminus \{i,j\}}(\varphi_k(\VEC{x}(k),\VEC{\bar y}^k))^{\VEC{a}(k)}
    \right]
    \VEC{a}(j) (\varphi_j(\VEC{x}(j), \VEC{\bar y}^j))^{\VEC{a}(j) -1}
    \frac{\partial\varphi_j}{\partial \VEC{y}(i)}(\VEC{x}(j),{\VEC{\bar y}}^j) .
\end{split}
\end{equation*}
By \eqref{eq: pgfTM_i}, for $i\neq j$,
\begin{equation}\label{eq: partialy_j}
		\frac{\partial\varphi_j}{\partial \VEC{y}(i)}(\VEC{x}(j),{\VEC{\bar y}}^j)
        = \VEC{x}(j)\frac{\partial g_j}{\partial \VEC{s}(j)}(\VEC{s}(j),{\VEC{\bar y}}^j)
        \frac{\partial\varphi_j}{\partial \VEC{y}(i)}(\VEC{x}(j),{\VEC{\bar y}}^j)
        + \VEC{x}(j)\frac{\partial g_j}{\partial \VEC{y}(i)}(\VEC{s}(j),{\VEC{\bar y}}^j),
\end{equation}
where $\VEC{s}(j)= \varphi_j(\VEC{x}(j),{\VEC{\bar y}}^j)$, and hence
\begin{equation}\label{eq: partvarphijtoi}
    \frac{\partial\varphi_j}{\partial \VEC{y}(i)}(\VEC{x}(j),{\VEC{\bar y}}^j)
        =\VEC{x}(j)\frac{\partial g_j}{\partial \VEC{y}(i)}(\VEC{s}(j),{\VEC{\bar y}}^j)
        \left[  1- \VEC{x}(j)\frac{\partial g_j}{\partial \VEC{s}(j)}(\VEC{s}(j),{\VEC{\bar y}}^j) \right]^{-1}.
\end{equation}
Evaluating at $\VEC{x}=\VEC{y}=\VEC{1}$ and using \eqref{eq: expiipgf} together with
$
    \mathbb{E}_{\VEC{e}_j}^r[ \boldsymbol{\xi}^{j}(i) ]
    =\frac{\partial g_j}{\partial \VEC{s}(i)}( \VEC{s}(j),\VEC{\bar s}^j)\big|_{\VEC{s}=\VEC{1}}=\frac{r}{d-1}$ for $ i\neq j$,
we deduce
\[
 \mathbb{E}_{\VEC{a}}^r\left[\VEC{M}_1(i)\right]
 = \sum_{j\neq i} \VEC{a}(j)
 \frac{\E_{\VEC{e}_j}^r\left[ \boldsymbol{\xi}^{j}(i) \right]}
 {  1-  \E_{\VEC{e}_j}^r\left[ \boldsymbol{\xi}^{j}(j) \right]}
 = \sum_{j\neq i} \frac{\VEC{a}(j)}{d-1} .
\]

 \item [(ii)]
Taking $\VEC{a}=\VEC{e}_i$ in (i) gives $\E^r_{\VEC{e}_i}[\VEC{M}_1(j)]=\frac{1}{d-1}$ for $j \neq i$ and $\E^r_{\VEC{e}_i}[\VEC{M}_1(i)]=0$, whence
\[
\mathbb{E}_{\VEC{e}_i}^r\left[|\VEC{M}_1|\right]
= \sum_{j\neq i} \mathbb{E}_{\VEC{e}_i}^r\left[\VEC{M}_1(j)\right] =1 .
\]
For the second moment, note that for each $i\in[d]$,
\begin{align*}
   \mathbb{E}_{\VEC{e}_i}^r\left[|\VEC{M}_1|^2\right]
    & =
    \sum_{k \neq i}
    \left[ \frac{\partial^2\varphi_i}{\partial \VEC{y}(k)^2}(x,\VEC{\bar y}^i)\Big|_{x=1,\VEC{y}=\VEC{1}} +
   \mathbb{E}_{\VEC{e}_i}^r\left[ \VEC{M}_1(k)  \right] \right]
    +  \sum_{\substack{j,k\neq i\\ j\neq k}}
    \frac{\partial^2\varphi_i}{\partial \VEC{y}(j)\partial \VEC{y}(k)}(x,\VEC{\bar y}^i)\Big|_{x=1,\VEC{y}=\VEC{1}}.
\end{align*}
Differentiating \eqref{eq: partialy_j} once more, for $j,k \neq i$,
\begin{align*}
\frac{\partial^2\varphi_i}{\partial \VEC{y}(j)\partial \VEC{y}(k)}(x,\VEC{\bar y}^i)
=& x\dfrac{\frac{\partial^2 g_i}{\partial \VEC{y}(k)\partial \VEC{s}(i)}
( \VEC{s}(i)  , {\VEC{\bar y}}^i )\,
\frac{\partial\varphi_i}{\partial \VEC{y}(j)}(x,{\VEC{\bar y}}^i)
 + \frac{\partial^2 g_i}{\partial \VEC{y}(j)\,\partial \VEC{y}(k)}
(  \VEC{s}(i)  , {\VEC{\bar y}}^i )}{1- x\frac{\partial g_i}{\partial \VEC{s}(i)}
(  \VEC{s}(i)  , {\VEC{\bar y}}^i )  },
\end{align*}
with $\VEC{s}(i)=\varphi_i(x,\VEC{\bar y}^i)$, and similarly for $\partial^2\varphi_i/\partial\VEC{y}(k)^2$. Evaluating at $x=1$, $\VEC{y}=\VEC{1}$ and summing, we obtain
\begin{align*}
   \mathbb{E}_{\VEC{e}_i}^r  \left[|\VEC{M}_1|^2\right]
&= \mathbb{E}_{\VEC{e}_i}^r\left[|\VEC{M}_1|\right]\frac{\mathbb{E}_{\VEC{e}_i}^r\left[\boldsymbol{\xi}^i(i)\sum_{k\neq i}\boldsymbol{\xi}^i(k)\right]}{1-\mathbb{E}_{\VEC{e}_i}^r\left[\boldsymbol{\xi}^i(i)\right] } + \frac{\mathbb{E}_{\VEC{e}_i}^r\left[\left(\sum_{k\neq i}\boldsymbol{\xi}^i(k)\right)^2\right]}{1-\mathbb{E}_{\VEC{e}_i}^r\left[\boldsymbol{\xi}^i(i)\right] }.
\end{align*}
Since $\E[(\xi^{(+)})^2]=\E[(\sum_{k=1}^d\boldsymbol{\xi}^i(k))^2]$ and $\E^r_{\VEC{e}_i}[\boldsymbol{\xi}^i(i)]=1-r\neq1$, the last identity shows that $\E[|\VEC{M}_1|^2]<\infty$ if and only if $\E[(\xi^{(+)})^2]<\infty$, which holds by \ref{hyp:size}.
\end{enumerate}
\hfill $\square$

\section*{Acknowledgements}
MCF's research is supported by the UNAM-PAPIIT grant IN109924. SHT acknowledges the support of the UNAM-PAPIIT grant IA103724.


\begin{thebibliography}{99}

\bibitem{Bertoin09}
J.~Bertoin, \textit{The structure of the allelic partition of the total population for Galton--Watson processes with neutral mutations}, Ann. Probab. \textbf{37} (2009), no.~4, 1502--1523.

\bibitem{Bertoin10}
J.~Bertoin, \textit{A limit theorem for trees of alleles in branching processes with rare neutral mutations}, Stoch. Process. their Appl. \textbf{120} (2010), no.~5, 678--697.

\bibitem{billingsley1999}
P.~Billingsley, \textit{Convergence of probability measures}, 2nd ed., Wiley Series in Probability and Statistics, John Wiley \& Sons, New York, 1999.

\bibitem{BFH}
A.~Blancas, M.~C. Fittipaldi, and S.~Hernández-Torres, \textit{Crossing bridges between percolation models and Bienaymé--Galton--Watson trees}, XIV Symposium on Probability and Stochastic Processes, Progress in Probability, vol.~81, Birkhäuser, Cham, 2025, pp.~103--132.

\bibitem{BlancasRivero}
A.~Blancas and V.~Rivero, \textit{On branching process with rare neutral mutation}, Bernoulli \textbf{24} (2018), no.~2, 1576--1612.

\bibitem{boxmaall-multJirina}
O.~Boxma, J.~Ivanovs, K.~Kosiński, and M.~Mandjes, \textit{Lévy-driven polling systems and continuous-state branching processes}, Stoch. Syst. \textbf{1} (2011), no.~2, 411--436.

\bibitem{Loic2016codmult}
L.~Chaumont and R.~Liu, \textit{Coding multitype forests: application to the law of the total population of branching forests}, Trans. Am. Math. Soc. \textbf{368} (2016), no.~4, 2723--2747.

\bibitem{Chauvin}
B.~Chauvin, \textit{Arbres et processus de Bellman--Harris}, Ann. Inst. Henri Poincaré Probab. Stat. \textbf{22} (1986), no.~2, 209--232.

\bibitem{chauvin1986propriete}
B.~Chauvin, \textit{Sur la propriété de branchement}, Ann. Inst. Henri Poincaré Probab. Stat. \textbf{22} (1986), no.~2, 233--236.

\bibitem{Dwass69}
M.~Dwass, \textit{The total progeny in a branching process and a related random walk}, J. Appl. Probab. \textbf{6} (1969), no.~3, 682--686.

\bibitem{HU2017109}
Z.~Hu, R.~Sun, and C.~Curtis, \textit{A population genetics perspective on the determinants of intra-tumor heterogeneity}, Biochim. Biophys. Acta Rev. Cancer \textbf{1867} (2017), no.~2, 109--126.

\bibitem{jagers2008}
P.~Jagers and A.~Lagerås, \textit{General branching processes conditioned on extinction are still branching processes}, Electron. Commun. Probab. \textbf{13} (2008), 540--547.

\bibitem{jivrina1958}
M.~Jiřina, \textit{Stochastic branching processes with continuous state space}, Czech. Math. J. \textbf{8} (1958), no.~2, 292--313.

\bibitem{kallenberg2002}
O.~Kallenberg, \textit{Foundations of modern probability}, 2nd ed., Probability and its Applications, Springer, New York, 2002.

\bibitem{kallenbergRM}
O.~Kallenberg, \textit{Random measures, theory and applications}, Springer, Cham, 2017.

\bibitem{Kimura1971}
M.~Kimura, \textit{Theoretical foundation of population genetics at the molecular level}, Theor. Popul. Biol. \textbf{2} (1971), no.~2, 174--208.

\bibitem{CellPhy}
A.~Kozlov, J.~M. Alves, A.~Stamatakis, and D.~Posada, \textit{CellPhy: accurate and fast probabilistic inference of single-cell phylogenies from scDNA-seq data}, Genome Biol. \textbf{23} (2022), no.~37.

\bibitem{PMID:29030470}
J.~Kuipers, K.~Jahn, B.~J. Raphael, and N.~Beerenwinkel, \textit{Single-cell sequencing data reveal widespread recurrence and loss of mutational hits in the life histories of tumors}, Genome Res. \textbf{27} (2017), no.~11, 1885--1894.

\bibitem{11GRAND}
D.~Lähnemann, J.~Köster, E.~Szczurek, et~al., \textit{Eleven grand challenges in single-cell data science}, Genome Biol. \textbf{21} (2020), no.~31.

\bibitem{MSRM13}
L.~A. Mathew, P.~R. Staab, L.~E. Rose, and D.~Metzler, \textit{Why to account for finite sites in population genetic studies and how to do this with Jaatha 2.0}, Ecol. Evol. \textbf{3} (2013), no.~11, 3647--3662.

\bibitem{pitman2006combinatorial}
J.~Pitman, \textit{Combinatorial stochastic processes: École d'été de probabilités de Saint-Flour XXXII-2002}, Lecture Notes in Mathematics, vol.~1875, Springer, Berlin, 2006.

\bibitem{resnick2007}
S.~I. Resnick, \textit{Heavy-tail phenomena: Probabilistic and statistical modeling}, Springer, New York, 2007.

\bibitem{whitt2002}
W.~Whitt, \textit{Stochastic-process limits}, Springer Series in Operations Research, Springer, New York, 2002.

\bibitem{SiCloneFit}
H.~Zafar, N.~Navin, K.~Chen, and L.~Nakhleh, \textit{SiCloneFit: Bayesian inference of population structure, genotype, and phylogeny of tumor clones from single-cell genome sequencing data}, Genome Res. \textbf{29} (2019), no.~11, 1847--1859.

\bibitem{SiFit17}
H.~Zafar, A.~Tzen, N.~Navin, K.~Chen, and L.~Nakhleh, \textit{SiFit: inferring tumor trees from single-cell sequencing data under finite-sites models}, Genome Biol. \textbf{18} (2017), no.~178.

\end{thebibliography}
\end{document}